\documentclass[11pt,reqno]{amsart}
\usepackage[left=0.9in,top=0.9in,right=0.9in,bottom=0.9in]{geometry}

\usepackage{amsfonts,amsmath,amsthm,amssymb,latexsym,mathrsfs,stmaryrd}
\usepackage{hyperref}
\usepackage{mathtools}
\usepackage{graphicx}
\usepackage{subcaption}
\usepackage{enumerate}
\usepackage{cleveref}

\usepackage{tikz}\usetikzlibrary{cd,fit,matrix,arrows,decorations.pathmorphing}
\tikzset{commutative diagrams/.cd}

\numberwithin{equation}{section}

\newtheorem{theorem}{Theorem}[section]
\newtheorem{corollary}[theorem]{Corollary}
\newtheorem{lemma}[theorem]{Lemma}
\newtheorem{proposition}[theorem]{Proposition}

\theoremstyle{definition}
\newtheorem{definition}{Definition}[section]
\newtheorem{definition-theorem}[definition]{Definition-Theorem}

\newtheorem{claim}[definition]{Claim}

\newtheorem{rmk}[definition]{Remark}
\newtheorem{remark}[definition]{Remark}

\theoremstyle{remark}

\newtheorem*{remark*}{Remark}

\newcommand\Z{{\mathbb Z}}

\newcommand\Q{{\mathbb Q}}

\newcommand\R{{\mathbb R}}
\newcommand\F{{\mathbb F}}
\def\P{{\mathbb P}}
\newcommand\Fp{{\mathbb F}_p}

\newcommand\Zp{{\mathbb Z}_p}

\DeclareMathOperator{\im}{im}

\DeclareMathOperator{\Hom}{Hom}

\DeclareMathOperator{\Aut}{Aut}

\DeclareMathOperator{\GL}{GL}

\DeclareMathOperator{\Mat}{Mat}

\newcommand\subeq{\subseteq}

\newcommand\onto{\twoheadrightarrow}

\DeclarePairedDelimiter{\abs}{\lvert}{\rvert}

\DeclarePairedDelimiter{\set}{\{}{\}}

\DeclarePairedDelimiter{\parens}{\lparen}{\rparen}

\newcommand{\FinMod}{\mathbf{FinMod}}

\newcommand{\Sur}{\mathrm{Sur}} 

\newcommand{\surj}{\twoheadrightarrow}

\newcommand{\cok}{\mathrm{cok}}
\newcommand{\Fl}{\mathbf{Fl}}
\newcommand{\bfcok}{\mathbf{cok}}
\newcommand{\calFl}{\mathcal{F}l}

\DeclareMathOperator{\corank}{corank}
\renewcommand{\sp}{\mathrm{span}}

\newcommand{\E}{\mathbb{E}}

\newcommand\rank{{\operatorname{rank}}}

\renewcommand\P{{\mathbf{P}}}
\newcommand{\Prob}{\P}

\newcommand\diag{{\operatorname{diag}}}

\newcommand\al{\alpha}
\newcommand\la{\lambda}

\newcommand\Bv{{\mathbf v}}

\renewcommand\Pr{{\mathbf P }}

\newcommand\CA{{\mathcal A}}

\newcommand\CC{{\mathcal C}}

\newcommand\CF{{\mathcal F}}

\newcommand\BBZ {{\mathbb Z}}

\newcommand{\tM}{\tilde{M}}

\newcommand\eps{\varepsilon}

\newcommand\lang{\langle}
\newcommand\rang{\rangle}

\newcommand{\sqbinom}[2]{\begin{bmatrix}#1\\ #2\end{bmatrix}}

\newcommand{\defn}{\textbf}
\renewcommand{\vec}{\mathbf}

\begin{document}
\title[Flag universality]{Cohen-Lenstra flag universality for random matrix products}

\author{Yifeng Huang}
\address{Department of Mathematics, University of Southern California}  
\email{yifenghu@usc.edu}

 \author{Hoi H. Nguyen}
 \address{Department of Mathematics, Ohio State University}
 \email{nguyen.1261@osu.edu}
 
 \author{Roger Van Peski}
\address{Department of Mathematics, Columbia University}  
\email{rv2549@columbia.edu}


\begin{abstract}
    For $n \times n$ random integer matrices $M_1,\ldots,M_k$, the cokernels of the partial products $\cok(M_1 \cdots M_i), 1 \leq i \leq k$ naturally define a random flag of abelian $p$-groups. We prove that as $n \to \infty$, this flag converges universally, for any nondegenerate entry distribution, to the Cohen-Lenstra type measure which weights each flag inversely proportional to the size of its automorphism group. As a corollary, we prove universality of certain formulas for the limiting conditional distribution of $\cok(M_1M_2)$ given $\cok(M_1),\cok(M_2)$ in terms of Hall-Littlewood structure constants, which were previously obtained only for Haar matrices over $\mathbb{Z}_p$.

    Our proofs combine the general technology of Sawin-Wood \cite{sawin2022moment}, matrix product moment computations following those of \cite{nguyen2022universality}, and the computation done previously for Haar $p$-adic matrices in \cite{huang2024cokernel}.
\end{abstract}

\maketitle

\tableofcontents

\section{Introduction}\label{sec:intro}

The classical \emph{Cohen-Lenstra distribution} is the probability distribution on (isomorphism classes of) finite abelian $p$-groups\footnote{i.e. $\abs{G}$ is a power of $p$.} such that the probability of each such $G$ is inversely proportional to the number of group automorphisms of $G$,
\begin{equation}
\Pr(G) = \frac{\prod_{i \geq 1}(1-p^{-i})}{\abs{\Aut(G)}}.
\end{equation}
Cohen and Lenstra \cite{cohen-lenstra} originally conjectured that it governs the $p$-Sylow subgroups---often also called $p^\infty$-torsion parts or $p$-parts---of class groups of quadratic imaginary number fields. It is now believed to also govern other random groups: for instance, M{\'e}sz{\'a}ros \cite{meszaros2020distribution} showed it governs asymptotics of the $p$-parts of Jacobians of Erd\H{o}s-R\'enyi directed graphs, and Kahle-Newman \cite{kahle2022topology} conjectured that it governs homology of certain random simplicial complexes. 

The reason for these results and conjectures comes from random matrix universality. For any $n$, one may take $M \in \Mat_n(\Z)$ with independent, identically distributed (iid) entries and consider its \emph{cokernel}
\begin{equation}
\cok(M) := \Z^n/M\Z^n,
\end{equation}
a random abelian group. Wood \cite{W1} showed that the $p$-Sylow subgroup $\cok(M)_p$ of this group will be asymptotically Cohen-Lenstra-distributed as $n \to \infty$, for any entry distribution within a broad class. This extended earlier results of Friedman-Washington \cite{friedman-washington}, who showed the same result for matrices with entries distributed by the additive Haar measure on the $p$-adic integers $\Z_p$. Those results relied on structure specific to that measure, while \cite{W1} used different techniques to show universality. Such universality lies behind the above-mentioned conjectures and theorems, since in each case there is a random (or pseudorandom) matrix in the background, but the distribution is not the one studied in \cite{friedman-washington}. Our object in the present work is to prove a much more general universality statement, `flag universality', involving cokernels of products of several random matrices and maps between them. It implies as corollaries the result of \cite{W1}, as well as later results of \cite{nguyen2022universality}. Another corollary of interest is the universality of the `conditional convolution' of two random matrices, \Cref{thm:conv_cor_intro}, the distribution of which was previously described using Hall-Littlewood polynomials (special functions on $p$-adic groups) in \cite{van2020limits}.

These asymptotics of \cite{W1} are insensitive to details of the matrix entry distribution, but certain matrix operations will change them. For instance, a polynomial such as $A^2+I$ in such a random matrix will have a different limiting cokernel distribution, shown in the above generality by Cheong-Yu \cite{cheong2023cokernel} with various related results by Cheong-Huang \cite{cheong2021cohen,cheong2023cokernel}, Cheong-Kaplan \cite{cheong2022generalizations}, Cheong-Liang-Strand \cite{cheong2023polynomial}, and Lee \cite{lee2023joint}. The distributions in these works no longer weight the cokernel by $1/\abs{\Aut(G)}$, but they are still `$1/\Aut$-type' distributions. The key is that the $p$-part of the cokernel of, for example, $A^2+I$, carries extra structure. It is not just an abelian $p$-group (equivalently, module over the $p$-adic integers $\Z_p$), but a $\Z_p[T]/(T^2+1)$-module where $T$ acts by $A$. The limiting distribution on such modules $H$ still gives them probability proportional to $1/\abs{\Aut(H)}$, but now these are module automorphisms rather than group automorphisms. The limiting distribution of the cokernel as a group will then be a marginal distribution of this one, obtained by summing over possible modules structures on the group, which in some sense obscures the natural structure. 

More generally, it is natural to consider multivariate polynomials in several matrices. Linear polynomials in two matrices were considered by Lee \cite{lee2023joint,lee2024mixed}. For nonlinear polynomials in many matrices, arguably the simplest example is a product $A_k \cdots A_2 A_1$ of independent random matrices. Universality of the limiting joint distribution of $\cok(A_1)_p,\cok(A_2A_1)_p,\ldots,\cok(A_k \cdots A_1)_p$, for the same matrix distributions as in \cite{W1}, was shown by the second and third authors \cite{nguyen2022universality}, following identification of the limit in \cite{van2020limits,vanpeski2021halllittlewood} for the same Haar `uniform' random matrices of \cite{friedman-washington}. We find that the most natural way of viewing this distribution, conjectured in \cite{nguyen2022universality} and proven in \Cref{thm:flag}, is by including the data of natural maps between these groups as in the above discussion.

\subsection{Setup and main result} Our results, and the previous universality results mentioned, apply in the following generality:

\begin{definition}\label{def:alpha_p_balanced}
Given real $\alpha \in (0,1/2]$ and a prime $p$, we say a $\Z$-valued random variable $\xi$ is \defn{$\al,p$-balanced} if
\begin{equation}\label{eqn:alpha_R}
\max_{r \in \Z/p\BBZ} \P(\xi\equiv r \pmod{p}) \le 1-\al.
\end{equation}
We say that a random matrix with $\Z$-valued entries is $\alpha,p$-balanced if the entries are $\alpha,p$-balanced and independent (but not necessarily iid). For a finite set of primes $P$, we say a $\Z$-valued random variable $\xi$ is \defn{$\al,P$-balanced} if it is $\al,p$-balanced for all $p\in P$, and refer to a matrix as $\alpha,P$-balanced if its entries are independent and $\alpha,P$-balanced.
\end{definition}

To understand the limit distribution there, it is helpful---as with polynomials in a single matrix---to consider the extra structure these groups have. Fix $k\in \Z_{\geq 1}$ and a finite set of primes $P$. 

Let $\FinMod_P$ be the category of finite abelian groups whose order contains only prime factors in $P$. We call an abelian group in $\FinMod_P$ an \defn{abelian $P$-group}. Equivalently, $\FinMod_P$ is the category of finite modules over the ring 
\begin{equation}\label{def:ZP}
    \Z_{(P)}:=\Z[s^{-1}: \gcd(s,\prod_{p\in P}p)=1].
\end{equation}
When $P=\set{p}$ is a single prime, we simply denote $P$ by $p$ instead of $\set{p}$, so our notation agrees with the standard notation with a single prime. For any finite abelian group $G$, we denote $G_P:=G\otimes_\Z \Z_{(P)}$, namely, the product of $p$-Sylow subgroups of $G$ for all $p\in P$.

 Let $\Fl_{k,P}$ be the category of \defn{surjective $k$-flags} of finite abelian $P$-groups, namely, chains of $k$ finite abelian $P$-groups connected by surjections
\begin{equation}
    \mathbf{G}=(G_k\onto \dots \onto G_1),\, G_i\text{ is a finite abelian $P$-group}.
\end{equation}
A morphism $\boldsymbol{\alpha}:\mathbf{G}\onto \mathbf{G}'$ between two surjective $k$-flags of finite $P$-groups is a sequence of group homomorphisms
\begin{equation}
    \boldsymbol{\alpha}=(\alpha_k,\dots,\alpha_1), \,\alpha_i\in \Hom(G_i,G'_i)
\end{equation}
such that the diagram
\begin{equation}
\begin{tikzcd}
    {G_k} & \cdots & {G_1} \\
    {G'_k} & \cdots & {G_1'}
    \arrow[from=1-1, to=1-2]
    \arrow["{\alpha_k}", from=1-1, to=2-1]
    \arrow[from=1-2, to=1-3]
    \arrow["{\alpha_1}", from=1-3, to=2-3]
    \arrow[from=2-1, to=2-2]
    \arrow[from=2-2, to=2-3]
\end{tikzcd}
\end{equation}
commutes. For $n\in \Z_{\geq 1}$ and $M_1,\dots,M_k\in \Mat_n(\Z)$, let $\bfcok(M_1,\dots,M_k)_P$ denote the surjective $k$-flag of abelian $P$-groups
\begin{equation}
    \cok(M_1\cdots M_k)_P \onto \dots \onto \cok(M_1)_P,
\end{equation}
with maps induced from the natural inclusion $\im(M_1\cdots M_i)\subeq \im(M_1\cdots M_{i-1})$. For random matrices, one obtains a random surjective flag. Our main result shows that it follows a universal distribution asymptotically:

\begin{theorem}[Flag universality]\label{thm:flag}
Fix $\alpha \in (0,1/2]$ and $P$ a finite set of primes, and for each $n$ let $M_1,\ldots,M_k$ be independent $n \times n$ $\alpha,P$-balanced matrices over $\Z$. Let $\mathbf{G}=(G_k\onto \dots \onto G_1)$ be a surjective $k$-flag of finite abelian $P$-groups. Then
    \begin{equation}\label{eq:main_flag_intro}
        \lim_{n\to \infty}\Prob(\bfcok(M_1,\dots,M_k)_P\simeq \mathbf{G}) = \frac{1}{\abs{\Aut_{\Fl_{k,P}}(\mathbf{G})}} \parens*{\prod_{p\in P}\prod_{i=1}^\infty (1-p^{-i})}^k,
    \end{equation}
    where $\Aut_{\Fl_{k,P}}(\mathbf{G})$ is the automorphism group of $\mathbf{G}$ as a flag.
\end{theorem}

\subsection{Previous work and proof outline} This result follows a chain of special cases. For instance, in previous work the latter two authors showed:

\begin{theorem}[Combining\footnote{Minor note: those results actually stated a more restrictive definition of $\alpha$-balanced random matrix, but remain true with the same proofs in the generality here.} {\cite[Theorems 1.2 and 1.3]{nguyen2022universality}}]\label{thm:from_nvp}
In the setting of \Cref{thm:flag}, the marginal distribution of isomorphism types of the $k$-tuple of groups $\cok(M_1)_P,\cok(M_1M_2)_P,\ldots,\cok(M_1\cdots M_k)_P$ converges to the marginal distribution on isomorphism types $(G_1,\ldots,G_k)$ of the probability distribution on the right-hand side of \eqref{eq:main_flag_intro}.
\end{theorem}

\cite[Theorem 1.2]{nguyen2022universality} also gave a more `explicit' form of this distribution in terms of surjections between $G_1,\ldots,G_k$ and automorphisms of these groups, but the above formulation was the most structurally natural and the most closely in parallel with the original Cohen-Lenstra distribution. This formulation gave strong reason to suspect \Cref{thm:flag}, which was conjectured at the end of Section 10 of \cite{nguyen2022universality}. For a very nice matrix model where explicit computations were possible, the second author subsequently proved the following:

\begin{theorem}[{\cite[Second part of Theorem 1.1]{huang2024cokernel}}]
The conclusion of \Cref{thm:flag} is true if instead $P=\set{p}$ is a single prime and $M_1,\ldots,M_k$ have entries in the $p$-adic integers $\Z_p$ which are distributed by the additive Haar measure.
\end{theorem}

The methods and results in these previous works also feature in our proof of \Cref{thm:flag}. The main tool to prove universality of cokernels of random matrices (as groups) is the \emph{moment method} developed by Wood \cite{W0}, which computes limits of $\E[\#\Sur(\cok(A),H)]$ where $H$ is a fixed group, and uses these limits to infer convergence in distribution of $\cok(A)_P$. The work \cite{nguyen2022universality} generalized this to joint distributions using joint moments $\E[\#\Sur(\cok(M_1 \cdots M_k),H_k) \cdots \#\Sur(\cok(M_1),H_1)]$ (as, independently, did Lee \cite{lee2024mixed}), but since these joint moments still only used group homomorphisms, they could only capture the group structure. 

Shortly after \cite{nguyen2022universality} appeared, Sawin and Wood \cite{sawin2022moment} gave a general moment method framework for random objects in so-called \emph{diamond categories}, which in particular includes abelian groups with various extra structures such as flags, pairings, etc. This has since been used in the aforementioned work on matrix polynomials \cite{cheong2023cokernel}, and to verify universality of Jacobian groups of random groups together with their canonical alternating pairings by Hodges \cite{hodges2023distribution}. Our proof combines
\begin{enumerate}
\item[(I)] The results of \cite{sawin2022moment}, which reduce the universality of the limit in \Cref{thm:flag} to moment computations (but do not actually compute explicitly what the limit is), \label{item:use_sw}
\vskip .05in
\item[(II)] Modifications of the computations in \cite{nguyen2022universality}, to actually compute the moments, and \label{item:compute_moments}
\vskip .05in
\item[(III)] The result of \cite{huang2024cokernel} in the $p$-adic Haar example, to explicitly identify the limit distribution. \label{item:find_limit}
\end{enumerate}
While our \Cref{thm:flag} is strictly more general than the results of \cite{nguyen2022universality}, our proofs are substantially shorter. This seems to be because the versions of moments needed for the machinery of \cite{sawin2022moment} are better aligned with the inductive structure of successive matrix products than the joint moments used in \cite{nguyen2022universality}.

\subsection{A corollary: universality of conditional convolution}\label{subsection:consequences}

Because the determinant is multiplicative and $\abs{\cok(M)} = \abs{\det(M)}$ for nonsingular $M \in \Mat_n(\Z)$, 
\begin{equation}
    \abs{\cok(M_1 M_2)_P} = \abs{\cok(M_1)_P} \cdot \abs{\cok(M_2)_P}.
\end{equation}

In particular, the cokernel $\cok(M_1 M_2)_P$ should not have the same asymptotic Cohen-Lenstra distribution since it is typically larger---and never smaller---than $\cok(M_1)$. One should view $\cok(M_1 M_2)_P$ a kind of `randomized convolution' of the two cokernels $\cok(M_1)_P$ and $\cok(M_2)_P$, which are themselves random. To understand this randomized convolution, it is natural to first condition on $\cok(M_1)_P$ and $\cok(M_2)_P$ and then consider the distribution of $\cok(M_1M_2)_P$, i.e. to consider the probabilities
\begin{equation}
    \Prob(\cok(M_1 M_2)_P \simeq G \mid \cok(M_1)_P \simeq H \text{ and }\cok(M_2)_P \simeq K)
\end{equation}
for any three abelian $P$-groups $G,H,K$. From the $k=2$ case of \Cref{thm:flag} we obtain:

\begin{corollary}\label{thm:conv_cor_intro}
    Let $M_1, M_2$ be independent $n \times n$ random matrices over $\Z$ with $\alpha,P$-balanced distribution in the sense of \Cref{def:alpha_p_balanced}, and let $H,K$ be any finite abelian $P$-groups. Then
    \begin{multline}\label{eq:conv_cor_intro}
        \lim_{n \to \infty} \Prob(\cok(M_1 M_2)_P \simeq G \mid \cok(M_1)_P \simeq H \text{ and }\cok(M_2)_P \simeq K) \\ 
        = \frac{\abs{\Aut(K)} \cdot \abs{\Aut(H)}}{\abs{\Aut(G)}} \abs{\{N \leq G: N \simeq K, G/N \simeq H\}}.
    \end{multline}
\end{corollary}

\begin{rmk}
    The probability in \eqref{eq:conv_cor_intro} is symmetric in $H$ and $K$. The left-hand side clearly is, since $\cok(M_1M_2) = \cok(M_2^TM_1^T)$ and the transpose matrices are also $\alpha,P$-balanced. The symmetry of the cardinality on the right-hand side is well-known and not difficult, see \cite[Chapter II]{mac}.
\end{rmk}

Reducing modulo $p$ yields universality of the conditional corank distribution for random matrices over finite fields, which is a bit more explicit:

\begin{corollary}
    \label{thm:rank_conv_intro} 
    Fix $\alpha \in (0,1/2]$. Let $M_1,M_2$ be independent $n \times n$ random matrices over $\F_p$, with entries that do not take any value with probability more than $1-\alpha$. Then for any $a,b,c \in \Z_{\geq 0}$,
    \begin{multline}
        \lim_{n \to \infty} \Prob\Big(\corank(M_1M_2) = c \mid  \corank(M_1)=a \text{ and }\corank(M_2) = b\Big) \\ 
        = p^{-(c-a)(c-b)}\frac{(p^{-1};p^{-1})_a(p^{-1};p^{-1})_b}{(p^{-1};p^{-1})_{c-b}(p^{-1};p^{-1})_{a+b-c}(p^{-1};p^{-1})_{c-a}},
    \end{multline}
    where $(p^{-1};p^{-1})_m := \prod_{i=1}^m (1-p^{-i})$ is the $q$-Pochhammer symbol, and we adopt the standard convention that $1/(p^{-1};p^{-1})_m=0$ if $m\in \Z_{<0}$, so that the right-hand side is zero unless $a,b\leq c \leq a+b$.
\end{corollary}

The case of a single prime $P=\{p\}$ in \Cref{thm:conv_cor_intro}, which lies between the general case and \Cref{thm:rank_conv_intro}, is particularly interesting. Any abelian $p$-group, including the cokernels $G,H,K$ in \Cref{thm:conv_cor_intro}, is isomorphic to a group 
\begin{equation}\label{eq:G_la}
G_\la := \bigoplus_{i=1}^n \Z/p^{\la_i}\Z.    
\end{equation}
These $p^{\la_i}$ also feature as diagonal entries in the Smith normal form of the matrix, an analogue of singular value decomposition, and it is natural to view $\la_1 \geq \la_2 \geq \ldots \geq \la_n \geq 0$ as analogues of singular values---see e.g. \cite{meszaros2024phase,van2020limits,van2023local}. In this analogy, the right-hand side of \eqref{eq:conv_cor_intro} may be seen as a discrete analogue of the finite free multiplicative convolution $\boxtimes$ in classical random matrix theory, and there are close parallels between the two at the level of formulas (see \Cref{section:formulas}). \Cref{thm:conv_cor_intro} is a type of `conditional convolution' universality whose complex analogue is easy to state, but still not established to the best of our knowledge, as we now explain.

\subsection{Structural connections to complex matrices}\label{subsec:structural_connections_to_complex_matrices}

Let $M$ be a random Hermitian matrix of size $n$, with real eigenvalues $\la_{1},\ldots, \la_{n}$. The empirical spectral distribution (ESD) of $M$ is the point measure 
$$\mu_{M} = \frac{1}{n} \sum_{i=1}^{n}\delta_{\la_{i}}.$$ 
Let $M_{1}, M_{2}$ be two independent matrices of size $n$ distributed by the GUE, so the entries are independent Gaussians subject to the Hermitian constraint. This is a very symmetric distribution which enjoys the following properties, not known for general Hermitian Wigner matrices: 

\begin{itemize}
  \item Condition on the ESD $\mu_{M_{1}}$ and $\mu_{M_{2}}$ of $M_{1}, M_{2}$ to be fixed measures, chosen such that they converge (weakly) to probability distributions $\mu_{1}, \mu_{2}$ on the real line. Then a celebrated result by Voiculescu \cite{Voi1} (see also \cite{Speicher} and \cite[Corollary 5.4.11]{AGZ}) states that the ESD of $M_{1}+M_{2}$ converges weakly almost surely to the additive free convolution  $\mu_{1} \boxplus_{} \mu_{2}$ of $\mu_{1}$ and $\mu_{2}$.
  \vskip .05in
  \item Similarly, condition that the eigenvalues of $M_{1}$ and $M_{2}$ are some fixed positive numbers, chosen so that as $n \to \infty$ the ESDs converge weakly to probability measures $\mu_{1}$ and $\mu_{2}$ on $\mathbb{R}_{\geq 0}$. Then the result of Voiculescu shows that the limiting empirical distribution of the singular values of $M_{1}M_{2}$ converges weakly almost surely to the multiplicative free convolution $\mu_{1} \boxtimes \mu_{2}$.
\end{itemize}
 For completeness, we refer the reader to Appendix \ref{section:formulas} for definitions of additive and multiplicative free convolutions. 

For fixed $n$, the distribution of the ESDs (equivalently, $n$-tuples of eigenvalues) of sums and products are also known. For $M_1,M_2$ GUE matrices conditioned on fixed eigenvalues $\vec{r}^{(1)}, \vec{r}^{(2)}$, the $n$-tuple of eigenvalues of $M_1+M_2$ and $M_1M_2$ follow `finite free convolution' laws $\delta_{\vec{r}^{(1)}} \boxplus_2 \delta_{\vec{r}^{(2)}}$ and\footnote{A minor point: in the multiplicative case, the indices $\vec{r}^{(i)}$ traditionally refer to logarithms of singular values rather than singular values themselves.} $\delta_{\vec{r}^{(1)}} \boxtimes_2 \delta_{\vec{r}^{(2)}}$ on $\R^{n}$, respectively. Both here and above, in the product case the matrices do not have to be Hermitian; for a matrix with iid complex Gaussian entries, one just considers singular values instead of eigenvalues and obtains the same results.

These finite free convolutions are prelimit versions of the free additive and multiplicative convolutions above, and are reasonably explicit, with formulas given in terms of multivariate Bessel and Heckman-Opdam functions by Gorin-Marcus \cite{GM}. Similarly, in the cokernel case, the limiting convolution operation we obtain in \Cref{thm:conv_cor_intro} is a large $n$ limit of exact formulas from symmetric function theory. Namely, there is an operation $\boxtimes_{HL,n}$ defined explicitly in terms of so-called Hall-Littlewood polynomials, such that
\begin{equation}\label{eq:finite_n_p-adic}
    \Pr\Big(\cok(M_1M_2)_p \cong G_\nu \mid \cok(M_1)_p \cong G_\la, \cok(M_2) \cong G_\mu\Big) = (\delta_{\la} \boxtimes_{HL,n} \delta_{\mu})(G_\nu)
\end{equation}
where $M_1,M_2$ are Haar-distributed $p$-adic matrices of size $n$ and $\delta_\cdot$ is the Dirac measure, see \Cref{thm:VP}. The explicit formulas are exact analogues of the ones for complex random matrices, both of which we recall in \Cref{section:formulas}. As $n \to \infty$, the random measure $\delta_{\la} \boxtimes_{HL,n} \delta_{\mu}$ converges to a random measure $\delta_{\la} \boxtimes_{HL} \delta_{\mu}$, which describes the limit distribution in \Cref{thm:conv_cor_intro}. We refer to \Cref{section:formulas} for details.

For matrices with different distribution from the ones above, in either the spectra or cokernel case, one does not see the same statistics at finite $n$; the results mentioned earlier are very specific to those distributions. In the $n \to \infty$ limit, while many universality results exist for matrices without conditioning, conditioning on spectra/cokernels as above makes the problem much harder outside of the invariant cases. In the complex case, to our knowledge no results of the above form have been extended to matrices with general random entries, and it is very unclear to us how to do so. For instance, following Voiculescu's results, one can consider independent Wigner matrices $M_{1}, M_{2}$ of size $n$ with general random entries, say of smooth density of mean zero and variance one. Conditioning on $M_{1}$ and $M_{2}$ having limiting ESDs $\mu_{1}$ and $\mu_{2}$ as $n\to \infty$, respectively, what is the limiting ESD of $M_{1} + M_{2}$? It is natural to expect that it is still $\mu_1 \boxplus \mu_2$ given known universality results for Wigner matrices, but does not seem easy to prove. A similar conjecture can be made for the product: under the same assumptions\footnote{In the product case, as discussed before, we condition on \( M_1 \) and \( M_2 \) being either general Wigner matrices or general iid matrices, whose limiting ESDs or limiting singular value distributions are \( \mu_1 \) and \( \mu_2 \), respectively, supported on \( \mathbb{R}_{\geq 0} \).}, the limiting singular value distribution of \( M_1 M_2 \) is expected to be \( \mu_1 \boxtimes \mu_2 \). 

The lack of such conditional convolution universality results for eigenvalues and singular values makes it all the more surprising that, in \Cref{thm:conv_cor_intro}, we obtain exactly such a universality result on the conditional convolution over $\Z$.

\textbf{Acknowledgments.} We thank Nicholas Cook, Jordan Ellenberg and Melanie Matchett Wood for helpful conversations. HN is supported by a Simons Travel Grant, TSM-00013318. RVP was partially supported by the European Research Council (ERC), Grant Agreement No. 101002013.

\subsection{Plan of paper}

\Cref{sec:diamond} introduces the algebraic setting necessary to apply the results of \cite{sawin2022moment}, This setup is then used in \Cref{sec:proof_and_state_moments} to reduce \Cref{thm:flag} to moment computations, which take most of the rest of the paper. In \Cref{sec:description_of_moments} we reduce to the case of random matrices over a finite ring $\Z/a\Z$, and state a number of basic lemmas for such matrices in \Cref{sec:supporting_lemmas}. The main moment computation is in \Cref{sec:moment_computation}, while \Cref{sec:verify_well_behaved} holds a verification of the `well-behaved' hypothesis necessary to apply \cite{sawin2022moment}. Finally, in \Cref{sec:prove_conv_cor} we specialize to the case of two matrices to prove \Cref{thm:conv_cor_intro} and deduce \Cref{thm:rank_conv_intro}. In \Cref{section:formulas} we give various definitions and formulas alluded to in the discussion of parallelism with complex matrix sums and products in \Cref{subsec:structural_connections_to_complex_matrices}.

\section{A diamond category of flags}\label{sec:diamond}
Let $P$ be a finite set of primes. The goal of this section is to define a subcategory $\CC_P$ of $\Fl_{k,P}$ that contains all objects and all isomorphisms in $\Fl_{k,P}$. We will perform the moment method of Sawin and Wood \cite{sawin2022moment} on $\CC_P$. It will turn out that the use of $\CC_P$ instead of $\Fl_{k,P}$ makes it easier to apply the techniques in \cite{sawin2022moment}.

\subsection{Injective flags}
For a finite abelian $P$-group $G$, an \textbf{injective $k$-flag} in $G$ is a sequence of nested subgroups
\begin{equation}
    \mathcal{F}=(H_1 \leq \dots \leq H_{k-1}\leq G).
\end{equation}
We say $H_i$ is the $i$-th \defn{step} of $\mathcal{F}$. Note that there are only $k-1$ steps $H_1,\dots,H_{k-1}$. By convention, we say the $0$-th step of $\mathcal{F}$ is $0$ and the $k$-th step of $\mathcal{F}$ is $G$. We denote by $s_i(\mathcal{F})$ the $i$-th step of $\mathcal{F}$, for $0\leq i\leq k$, which is useful when $H_i$ is not mentioned in the context.

Recall $\FinMod_P$ is the category of finite abelian $P$-groups. Consider the functor $\calFl_k:\FinMod_{P}\to \mathbf{Sets}$ that, object-wise, associates to an abelian $P$-group $G$ the set of injective $k$-flags in $G$. Morphism-wise, the functor $\calFl_k$ is defined as follows. For a homomorphism $\alpha:G\to G'$ of abelian $P$-groups and an injective $k$-flag 
\begin{equation}
    \mathcal{F}=(H_1 \leq \dots \leq H_{k-1}\leq G)\in \calFl_k(G),
\end{equation}
define $\alpha(\mathcal{F})\in \calFl_k(G')$ to be the injective $k$-flag $\mathcal{F}'$ 
\begin{equation}
    \mathcal{F}'=(\alpha(H_1)\leq \dots \leq \alpha(H_{k-1}) \leq G').
\end{equation}
Therefore, the $i$-th step of $\mathcal{F}'$ is
\begin{equation}
    s_i(\mathcal{F}')=\alpha(s_i(\mathcal{F}))
\end{equation}
for $0\leq i\leq k-1$. However, it is not true that $s_k(\mathcal{F}')=\alpha(s_k(\mathcal{F}))$ unless $\alpha$ is surjective, as $s_k(\mathcal{F}')=G'$ by convention. 

It is clear that $\calFl_k$ by our definition is a functor.

\subsection{The category of finite abelian $P$-groups together with an injective $k$-flag}

\begin{definition}\label{def:C}
Define $\CC_P$ to be the category $(\FinMod_P,\calFl_k)$ in the notation of \cite[Lemma~6.22]{sawin2022moment}. Namely, the objects of $\CC_P$ are pairs $(G,\mathcal{F})$, where $G\in \FinMod_P$ and $\mathcal{F}\in \calFl_k(G)$. A morphism $\alpha: (G,\mathcal{F})\to (G',\mathcal{F}')$ is a homomorphism $\alpha:G\to G'$ such that $\alpha(\mathcal{F})=\mathcal{F}'$. 
\end{definition}

\begin{lemma}\label{lem:epi}
    The category $\CC_P$ is a diamond category in the sense of \cite{sawin2022moment}. Moreover, a morphism $\alpha:(G,\mathcal{F})\to (G',\mathcal{F}')$ is an epimorphism in $\CC_P$ if and only if $\alpha$ is surjective. 
\end{lemma}
\begin{proof}
    By \cite[Lemma~6.1]{sawin2022moment} with $R=\Z_{(P)}$ (see \eqref{def:ZP}), the category $\FinMod_P$ is a diamond category. The first assertion then follows directly from \cite[Lemma~6.22]{sawin2022moment}. 
    
    The ``if'' part of the second assertion follows from the fact that the forgetful functor $\CC\to \FinMod_P$ is faithful (note that $\Hom_\CC((G,\CF),(G',\CF'))$ is a subset of $\Hom_{\FinMod_P}(G,G')$ by construction). The ``only if'' part of the second assertion follows from the proof of \cite[Lemma~6.22]{sawin2022moment}. More precisely, the first two paragraphs of that proof show that if a morphism $\alpha: G \to G'$ in $\FinMod_p$ such that $\alpha(\mathcal{F})=\mathcal{F}'$ is an epimorphism $(G,\mathcal{F}) \to (G',\mathcal{F}')$, then $\alpha$ is an epimorphism in $\FinMod_P$. 
\end{proof}

\subsection{Comparison with surjective flags}
We define a functor $Q: \CC_P\to \Fl_{k,P}$ as follows. Given $(G,\mathcal{F})\in \CC_P$ with
\begin{equation}
    \mathcal{F}=(H_1 \leq \dots \leq H_{k-1}\leq G)\in \calFl_k(G),
\end{equation}
define $Q((G,\mathcal{F}))$ to be the surjective $k$-flag
\begin{equation}
    G/\mathcal{F}:=(G\onto G/H_1\onto \dots \onto G/H_{k-1}).
\end{equation}
Given two objects $(G,\mathcal{F})$ and $(G',\mathcal{F}')$ in $\CC_P$, we note that a morphism from $G/\mathcal{F}\to G'/\mathcal{F}'$ in $\Fl_{k,P}$ is a homomorphism $\alpha:G\to G'$ such that the following commutative diagram factors for all $1\leq i\leq k-1$:
\begin{equation}
    \begin{tikzcd}
        G & G/s_i(\CF) \\
        G' & G'/s_i(\CF')\\
        \arrow[from=1-1, to=1-2]
        \arrow[from=2-1, to=2-2]
        \arrow["\alpha", from=1-1, to=2-1]
        \arrow[dashed, from=1-2, to=2-2]
    \end{tikzcd}
\end{equation}
This is equivalent to the condition that $\alpha(s_i(\CF))\subeq s_i(\CF')$ for all $1\leq i\leq k-1$. On the other hand, a morphism from $(G,\CF)$ to $(G',\CF')$ in $\CC_P$ is a homomorphism $\alpha:G\to G'$ such that $\alpha(s_i(\CF))=s_i(\CF')$ for all $1\leq i\leq k-1$. This shows that $Q$ is well-defined as a functor in the obvious way, together with
\begin{lemma}\label{lem:replace}
    The functor $Q$ is faithful and essentially surjective, identifying $\CC_P$ as a subcategory of $\Fl_{k,P}$ that contains every object up to isomorphism. Moreover, $\CC_P$ contains every isomorphism in $\Fl_{k,P}$.
\end{lemma}
\begin{proof}
    The faithfulness and the fact that $\CC_P$ contains every isomorphism in $\Fl_{k,P}$ follows from the description of Hom sets above. It only remains to show that $Q$ is essentially surjective. Given $\mathbf{G}=(G_k\onto \dots\onto G_1)\in \Fl_{k,P}$, define $H_i=\ker(G_k\onto G_{k-i})$ for $1\leq i\leq k-1$. Clearly, 
    \begin{equation}
        \CF=(H_1\leq \dots \leq H_{k-1}\leq G_k)
    \end{equation}
    forms an injective $k$-flag in $G_k$, and $Q((G_k,\CF))\simeq_{\Fl_{k,P}} \mathbf{G}$.
\end{proof}

\begin{remark}
    Our choice to work with $\CC_P$ rather than (the presumably more natural choice) $\Fl_{k,P}$ is out of technical convenience. It is not hard to show directly (or by \cite[Proposition~6.18]{sawin2022moment}) that $\Fl_{k,P}$ is a diamond category. However, we will use \cite[Lemma~6.23]{sawin2022moment} to quickly verify that the moment function we will find on $\CC_P$ is well-behaved (so their machinery works), and it is not immediate that this lemma is applicable to $\Fl_{k,P}$. As is remarked right after \cite[Lemma~6.23]{sawin2022moment}, analogue of this lemma need not hold for categories that fit into the framework of \cite[Proposition~6.18]{sawin2022moment} but do not fit into the framework of \cite[Lemma~6.22]{sawin2022moment}.
\end{remark}

From now on, we will interchangeably interpret an object in $\Fl_{k,P}$ (equivalently, an object in $\CC_P$) as a surjective flag $\mathbf{G}$ or a group together with an injective flag $(G,\CF)$, and write $\mathbf{G}=(G,\CF)$ if $\mathbf{G}=Q((G,\CF))$. However, we will use $\Hom_{\CC_P}$ and $\Sur_{\CC_P}$ to make clear that we are talking about morphisms (resp.~epimorphisms) in the category $\CC_P$.

\subsection{From single prime to multiple primes}
In order to draw results about random $p$-adic matrices for a single prime $p$ and apply to our situation with multiple primes $p$, we need some simple observations in \cite{sawin2022moment} about product categories. For two categories $C_1$ and $C_2$, let $C_1\times C_2$ be the \defn{product category} whose objects are ordered pairs of an object of $C_1$ and an object of $C_2$ and whose morphisms are ordered pairs of morphisms. By \cite[Lemma 6.16]{sawin2022moment}, if $C_1,C_2$ are diamond categories, then $C_1\times C_2$ is a diamond category.

\begin{lemma}\label{lem:cPcp}
    As categories, $\CC_P\simeq\prod_{p\in P} \CC_p$.
\end{lemma}
\begin{proof}
    Recall that for any $p\in P$, we have a functor $(\cdot)_p:\FinMod_P\to \FinMod_p$ defined as $(\cdot)_p = (\cdot)\otimes_\Z \Zp$. Objects of $\CC_P$ are injective $k$-flags of finite abelian $P$-groups
    \begin{equation}
        (G,\mathcal{F})=(H_1 \leq \dots \leq H_{k-1}\leq G).
    \end{equation}
    For $p\in P$, define
    \begin{equation}
        (G,\mathcal{F})_p=((H_1)_p \leq \dots \leq (H_{k-1})_p\leq G_p)
    \end{equation} 
    by taking $p$-Sylows on each step. Recall that a morphism $\alpha:(G,\mathcal{F})\to (G',\mathcal{F}')$ in $\CC_P$ is a morphism $\alpha:G\to G'$ of finite abelian $P$-groups such that $\alpha(s_i(\mathcal F))=\alpha(\mathcal F')$. For $p\in P$, define $\alpha_p:(G,\mathcal{F})_p\to (G',\mathcal{F}')_p$ by taking the $(\cdot)_p$ functor on each step. 
    
    Define the functor $\CC_P\to \prod_{p\in P} \CC_p$ object-wise by $(G,\mathcal{F})\mapsto ((G,\mathcal{F})_p)_{p\in P}$ and morphism-wise by $\alpha\mapsto (\alpha_p)_{p\in P}$. This is an equivalence of categories because given objects $(G,\mathcal{F})$ and $(G',\mathcal{F}')$ of $\CC_P$, a morphism $\alpha:G\to G'$ satisfies $\alpha(s_i(\mathcal F))=\alpha(\mathcal F')$ if and only if $\alpha_p(s_i(\mathcal F)_p)=\alpha_p(\mathcal F'_p)$ for all $p\in P$.
\end{proof}

The formalism of moments in \cite{sawin2022moment} behaves nicely with respect to products of diamond categories.
By a \defn{moment function} on any category $C$, we mean an assignment of a real number to each isomorphism class of objects in $C$. We denote a general moment function by $M=(M_G)_{G\in C/\simeq}\in \R^{C/\simeq}$. A moment function $(M_G)_G$ is \defn{nonnegative} if $M_G\in \R_{\geq 0}$ for all $G\in C/\simeq$. For a probability measure $\nu$ on $C/\simeq$, the \defn{moment function associated with $\nu$} is the moment function $M^\nu=(M^\nu_G)_G$ with
\begin{equation}
    M^\nu_G:=\int_{H\in C/\simeq}\abs*{\mathrm{Epi}_C(H,G)}\,d\nu.
\end{equation}
The moment function associated with $\nu$ exists if the integral defining $M^\nu_G$ converges for all $G\in C/\simeq$.

\begin{lemma}\label{lem:moment-product}
    Let $C_1,C_2$ be diamond categories, and $\nu_1,\nu_2$ be probability measures on $C_1,C_2$ respectively whose moment functions exist. Consider the product measure $\nu=\nu_1\times \nu_2$ on the product category $C=C_1\times C_2$. Then the moment function associated with $\nu$ exists, and
    \begin{equation}
        M^\nu_{(G_1,G_2)}=M^{\nu_1}_{G_1} M^{\nu_2}_{G_2},\quad \text{ for }G_1\in C_1/\simeq,\, G_2\in C_2/\simeq.
    \end{equation}
\end{lemma}
\begin{proof}
    Recall the standard fact that a morphism $(f_1,f_2)$ in $C_1\times C_2$ is an epimorphism if and only if both $f_1,f_2$ are epimorphisms. Then the result follows from Fubini's theorem:
    \begin{equation}
        \begin{aligned}
            M^\nu_{(G_1,G_2)} &=\int_{(H_1,H_2)\in C_1\times C_2/\simeq}\abs*{\mathrm{Epi}_{C_1\times C_2}((H_1,H_2),(G_1,G_2))}\,d\nu_1d\nu_2\\&=\int_{(H_1,H_2)\in C_1\times C_2/\simeq}\abs*{\mathrm{Epi}_{C_1}(H_1,G_1)\times \mathrm{Epi}_{C_2}(H_2,G_2)}\,d\nu_1d\nu_2\\
            &=\int_{H_1\in C_1/\simeq}\abs*{\mathrm{Epi}_{C_1}(H_1,G_1)}\,d\nu_1 \int_{H_2\in C_2/\simeq}\abs*{\mathrm{Epi}_{C_2}(H_2,G_2)}\,d\nu_2=M^{\nu_1}_{G_1} M^{\nu_2}_{G_2},
        \end{aligned}
    \end{equation}
    where each equality is in the sense that if the right-hand side converges then so does the left-hand side.
\end{proof}

For a diamond category $C$ and a moment function $M$ on $C$, we have the notion of \defn{well-behavedness} defined in \cite[p.~4]{sawin2022moment}, a condition saying $M$ does not grow too fast that is sufficient to guarantee that $M$ is the moment function associated with a unique (signed) measure, among other niceness properties. For the reader's benefit let us note that it will not matter what the actual definition is if one accepts the results in \cite{sawin2022moment}, as we appeal to results of \cite{sawin2022moment} to verify it and then use it as input to other results of \cite{sawin2022moment}.

\begin{lemma}[{\cite[Lemma 6.17]{sawin2022moment}}]
    Suppose a nonnegative moment function $M$ on a product of diamond categories $C_1\times C_2$ is \defn{separable}, i.e., there are nonnegative moment functions $M^1$ on $C_1$ and $M^2$ on $C_2$ such that $M_{(G_1,G_2)}=M^1_{G_1}M^2_{G_2}$ for all $(G_1,G_2)\in C_1\times C_2/\simeq$. If $M^1$ and $M^2$ are both well-behaved, then so is $M^1\times M^2$. \label{lem:well-product}
\end{lemma}
The final ingredient we need in order to apply the formalism of \cite{sawin2022moment} is a statement that our desired limiting moment is well-behaved.

\begin{proposition}\label{prop:well-behaved}
    For any prime $p$, consider the moment function $\boldsymbol{1}$ on $\CC_p$ that assigns $1$ to all $\mathbf{G}\in \CC_p$. Then $\boldsymbol{1}$ is well-behaved.
\end{proposition}

We will prove this later in \Cref{sec:verify_well_behaved}.

\begin{corollary}\label{cor:well-behaved-multiprime}
    For a finite set of primes $P$, consider the moment function $\boldsymbol{1}$ on $\CC_P$ that assigns $1$ to all $\mathbf{G}\in \CC_P$. Then $\boldsymbol{1}$ is well-behaved.
\end{corollary}
\begin{proof}
    Viewing $\CC_P$ as $\prod_{p\in P}\CC_p$ according to \Cref{lem:cPcp}, the moment function $\boldsymbol{1}_{\CC_P}$ is separable and is the product of the moment functions $\boldsymbol{1}_{\CC_p}$ for all $p\in P$. The result follows from \Cref{lem:well-product} and \Cref{prop:well-behaved}.
\end{proof}

\section{{Proof of \Cref{thm:flag} conditional on moment computations}}\label{sec:proof_and_state_moments}

The following theorem, which we will prove in \Cref{sec:description_of_moments,sec:moment_computation}, gives the desired limiting moments.

\begin{theorem}\label{thm:moment}
    Fix a surjective $k$-flag $\mathbf{G}\in \Fl_{k,P}$, $P$ a finite set of primes, and $\alpha \in (0,1/2]$. For each $n\in \Z_{\geq 1}$, let $M_1,\dots,M_k\in \Mat_n(\Z)$ be independent $\alpha,P$-balanced random matrices. Then 
    \begin{equation}
        \lim_{n\to \infty} \E_{M_1,\dots,M_k\in \Mat_n(\Z)}[\abs{\Sur_{\CC_P}(\bfcok(M_1,\dots,M_k)_P,\mathbf{G})}] =1.
    \end{equation}
\end{theorem}


To motivate the results of \Cref{prop:well-behaved} and \Cref{thm:moment} which we have just introduced, we show how they imply our main Theorem \ref{thm:flag}.

\begin{proof}[Proof of \Cref{thm:flag}]
    Viewing $\bfcok(M_1,\ldots,M_k)_P$ and $\mathbf{G}$ as elements of $\mathcal{C}_P$ by \Cref{lem:replace}, we have by the same lemma that the event $\bfcok(M_1,\ldots,M_k)_P \simeq_{\Fl_{k,P}} \mathbf{G}$ in \Cref{thm:flag} is equivalent to $\bfcok(M_1,\ldots,M_k)_P \simeq_{\mathcal{C}_P} \mathbf{G}$. So we work with $\mathcal{C}_P$ from now on.
    
    Now consider arbitrary random matrices $M_1,\ldots,M_k$ satisfying the assumptions of \Cref{thm:flag}. By \Cref{prop:well-behaved}, the moment function $\boldsymbol{1}_P:=\boldsymbol{1}_{\CC_P}$ on $\mathcal{C}_P$ is well-behaved. By \Cref{thm:moment}, the moment of $\bfcok(M_1,\ldots,M_k)_P$ converges to $\boldsymbol{1}_P$. Hence \cite[Theorem 1.8]{sawin2022moment} implies that the following are true:
    \begin{enumerate}
    \item There is a unique probability measure $\nu_P$ on $\mathcal{C}_P$ with moment $\boldsymbol{1}_P$.
  \vskip .05in
  \item $\bfcok(M_1,\ldots,M_k)_P$ converges in distribution to $\nu_P$ as $n \to \infty$. 
    \end{enumerate}
    Hence, it suffices to show that the measure $\nu_P$ associated with moment $\boldsymbol{1}_P$ is given by
    \begin{equation}\label{eq:limiting-dist}
        \nu_P(\mathbf{G})=\frac{1}{\abs{\Aut_{\Fl_{k,P}}(\mathbf{G})}} \parens*{\prod_{p\in P}\prod_{i=1}^\infty (1-p^{-i})}^k.
    \end{equation}

    While computing the distribution given the moment is a difficult task in general, we are able to avoid this task by using the result of \cite{huang2024cokernel} on Haar random matrices over $\Zp$, as we will do for the rest of the proof. The discussion above applied to the special case $P=\set{p}$ gives rise to a unique probability measure $\nu_p$ on $\CC_p$ with moment $\boldsymbol{1}_p$. Since $\boldsymbol{1}_P=\prod_{p\in P}\boldsymbol{1}_p$, by \Cref{lem:moment-product} the moment associated with the product measure $\prod_{p\in P}\nu_p$ on $\CC_P/\simeq$ is $\boldsymbol{1}_P$. By the uniqueness statement defining $\nu_P$, we conclude $\nu_P=\prod_{p\in P}\nu_p$, or in other words
    \begin{equation}
        \nu_P(\mathbf{G})= \prod_{p\in P} \nu_p(\mathbf{G}_p)\text{ for }\mathbf{G}\in \Fl_{k,P}.
    \end{equation}
    Since $\Aut_{\Fl_{k,P}}(\mathbf{G})=\prod_{p\in P} \Aut_{\Fl_{k,p}}(\mathbf{G}_p)$, to prove \eqref{eq:limiting-dist} it suffices to prove the special case $P=\set{p}$, which we do next.

    Let $\mathbf{G}=(G_k\onto\dots\onto G_1)\in \Fl_{k,p}$ be fixed. We note from the uniqueness part of statement (1) applied to the case $P=\set{p}$ that the measure $\nu_p$ does not depend on the random matrices $M_1,\ldots,M_k$. In particular, let $N$ be an integer such that $p^{N-1}G_k=0$, and let $M_1^{(N)},\ldots,M_k^{(N)}$ be iid matrices with iid entries distributed by the uniform measure on $\set{0,1,\dots,p^N-1}$, then statement (2) applied to this special case yields
    \begin{equation}\label{eq:proof-of-main}
        \lim_{n \to \infty} \Prob(\bfcok(M_1^{(N)},\ldots,M_k^{(N)})_p \simeq \mathbf{G}) = \nu_p(\mathbf{G}). 
    \end{equation}
    We claim that the left-hand side is unchanged if we replace $M_i^{(N)}$ by iid matrices $M_i'$ with iid entries distributed by the Haar measure on $\Zp$. Indeed, since $p^{N-1}G_k=0$, the two events below for a random flag $\mathbf{G'}$ of abelian $p$-groups are the same:
    \begin{equation}
        \set*{\mathbf{G'}\otimes \Z/p^N\Z \simeq \mathbf{G}} = \set*{\mathbf{G'} \simeq \mathbf{G}}.
    \end{equation}
    Since $M_i^{(N)}$ and $M_i'$ are identically distributed modulo $p^N$, and taking cokernels commute with taking tensor products, it follows that 
    \begin{equation}
        \Prob(\bfcok(M_1^{(N)},\ldots,M_k^{(N)})_p \simeq \mathbf{G}) = \Prob(\bfcok(M_1',\ldots,M_k')_p \simeq \mathbf{G}).
    \end{equation}
    On the other hand, \cite[Theorem 1.1]{huang2024cokernel} states that 
    \begin{equation}
    \lim_{n \to \infty} \Prob(\cok(M_1',\ldots,M_k') \simeq \mathbf{G}) = \frac{1}{\abs{\Aut_{\mathbf{Fl}_{k,p}}(\mathbf{G})}} \left(\prod_{i=1}^\infty (1-p^{-i})\right)^k.
    \end{equation}
    Combining the equalities, we get
    \begin{equation}
        \nu_p(\mathbf{G})=\frac{1}{\abs{\Aut_{\mathbf{Fl}_{k,p}}(\mathbf{G})}} \left(\prod_{i=1}^\infty (1-p^{-i})\right)^k,
    \end{equation}
    and the proof is complete.
\end{proof}

\section{Reducing \Cref{thm:moment} to spans over a finite ring}\label{sec:description_of_moments} 
In this section, we convert the moment computation problem in Theorem \ref{thm:moment} into a more explicit form, \Cref{thm:chain:span}, which is what we will prove subsequently.

\begin{definition}\label{def:R_alpha_balanced}
Let $a \in \Z_{\geq 1}$ and $R=\Z/a\Z$. Given real $\alpha \in (0,1/2]$, we say an $R$-valued random variable $\xi$ is \emph{$\al$-balanced} if for every prime $p|a$ we have
\begin{equation}\label{eqn:alpha_R}
\max_{r \in \Z/p\BBZ} \P(\xi\equiv r \pmod{p}) \le 1-\al.
\end{equation}
We say that a random matrix with $R$-valued entries is $\alpha$-balanced for some $\alpha \in (0,1/2]$ if the entries are independent and $\alpha$-balanced (but not necessarily iid).
\end{definition}

Note that for $P = \{p \text{ prime}: p\mid a\}$, trivially a $\Z$-valued random variable $\xi$ is $(\alpha,P)$-balanced if and only if the $R$-valued random variable $\xi \pmod{a}$ is $\alpha$-balanced.

\begin{proposition}\label{thm:chain:span}
Let $a \in \Z_{\geq 1}$, $R=\Z/a\Z$, $\alpha \in (0,1/2]$. For each $n$, let $M_1,\dots,M_k$ be independent and $\alpha$-balanced random matrices in $\Mat_n(R)$. Let $G$ be an $R$-module (i.e., finite $a$-torsion abelian group), and $0=H_0\leq H_1\leq \dots \leq H_{k-1}\leq H_k=G$ be an injective flag of subgroups. Then
    \begin{equation}\label{eq:moment_F_probs}
        \lim_{n\to\infty}\sum_{F\in G^n} \Prob\Big(\sp(M_i\cdots M_k F)=H_{i-1} \text{ for all $1\leq i\leq k+1$}\Big)=1.
    \end{equation}
    Here, we sum over all column vectors $F$ with entries in $G$, and the span of a column vector in $G^n$ means the subgroup of $G$ generated by the entries of the vector. 
\end{proposition}

\begin{proof}[Proof that \Cref{thm:chain:span} implies \Cref{thm:moment}]
    Letting $\tM_i$ be the matrices denoted $M_i$ in \Cref{thm:moment} (note that \Cref{thm:chain:span} uses $M_i$ differently, and in this proof we will reserve $M_i$ for the latter), we must show
    \begin{equation}\label{eq:chainspan_wts}
        \lim_{n\to \infty} \E_{\tM_1,\dots,\tM_k}[\abs{\Sur_{\CC_P}(\bfcok(\tM_1,\dots,\tM_k)_P,\mathbf{G})}] =1.
    \end{equation}
    Let $G$ be the largest group in the surjective $k$-flag $\bf{G}$ of \Cref{thm:moment}, and let $a \in \Z_{\geq 1}$ be an annihilator of $G$, i.e. $a g = 0$ for every $g \in G$. Let $R=\Z/a\Z$ and let $M_i := \tM_i \pmod{a}$, a random element of $\Mat_n(R)$. Since $\tM_1,\ldots,\tM_k$ were $(\alpha,P)$-balanced, $M_1,\ldots,M_k$ are $\alpha$-balanced. It is clear that 
    \begin{equation}\label{eq:tM_to_M}
        \abs{\Sur_{\CC_P}(\bfcok(\tM_1,\dots,\tM_k)_P,\mathbf{G})} = \abs{\Sur_{\CC_P}(\bfcok(M_1,\dots,M_k),\mathbf{G})},
    \end{equation}
    where here and below we use the same notation $\cok$ and $\bfcok$ for matrices over $R$ without comment. Hence to show \eqref{eq:chainspan_wts}, we need only work with the matrices $M_1,\ldots,M_k$.

Let $\CF_{M_1,\dots,M_k}\in \calFl_k(\cok(M_1\cdots M_k))$ be the injective flag such that $\bfcok(M_1,\dots,M_k)$ can be identified via Lemma \ref{lem:replace} with the object $(\cok(M_1\cdots M_k),\CF_{M_1,\dots,M_k})$ in $\CC_P$; we say $\CF_{M_1,\dots,M_k}$ is the injective flag associated with $\bfcok(M_1,\dots,M_k)$. Using the proof of Lemma \ref{lem:replace}, the steps of the flag $\CF_{M_1,\dots,M_k}$ are given by
\begin{equation}
    s_i(\CF_{M_1,\ldots,M_k})=\ker\parens*{\frac{R^n}{\im(M_1\cdots M_k)}\to \frac{R^n}{\im(M_1\cdots M_{k-i})}}=\frac{\im(M_1\cdots M_{k-i})}{\im(M_1\cdots M_k)}.
\end{equation}
Now consider a fixed object $\mathbb{G}=(G,\CF)$ of $\CC_P$, where $\CF$ is the flag $H_1\leq \dots \leq H_{k-1}\leq G$. By Lemma \ref{lem:epi} and the definition of $\CC_P$, an epimorphism $\alpha\in \Sur_{\CC_P}(\bfcok(M_1,\dots, M_k), \mathbb{G})$ (note importantly that we are in $\CC_P$, not $\Fl_{k,p}$) is a surjection $F:\cok(M_1\cdots M_k)\to G$ such that
\begin{equation}
    F\parens*{\frac{\im(M_1\cdots M_{k-i})}{\im(M_1\cdots M_k)}} = H_i
\end{equation}
for $1\leq i\leq k-1$. 

Since a homomorphism $F:\cok(M_1\cdots M_k)\to G$ is nothing but an $R$-linear map (denoted also by $F$ by abuse of notation) $F:R^n\to G$ such that $F(\im(M_1\cdots M_k))=0\in G$, the set $\Sur_{\CC_P}(\cok(M_1,\dots,M_k))$ can be equivalently described as
\begin{equation}
    \left\{ F\in \Hom(R^n, G) \;\middle| \begin{array}{l}
        \im(FM_1\cdots M_k)=0\\
        \im(FM_1\cdots M_{k-i}) = H_i \text{ for }1\leq i\leq k-1\\
        \im(F)=G
    \end{array} \right\},
\end{equation}
where the last condition takes care of the requirement that $F$ is surjective. Using the convention that $H_0=s_0(\CF)=0$ and $H_k=s_k(\CF)=G$, the above can be compactly rewritten as a canonical bijection
\begin{equation}\label{eq:sur-description}
    \Sur_{\CC_P}(\bfcok(M_1,\dots,M_k)) \leftrightarrow \set{F\in \Hom(R^n, G): \im(FM_1\cdots M_{k-i}) = H_i \text{ for }0\leq i\leq k}.
\end{equation}
It follows immediately from \eqref{eq:sur-description} that 
    \begin{multline}\label{eq:after_modawts}
        \lim_{n \to \infty} \E_{M_1,\dots,M_k}[\abs{\Sur_{\CC_P}(\bfcok(M_1,\dots,M_k)_P,\mathbf{G})}]  \\ = \lim_{n \to \infty} \sum_{F\in G^n} \Prob(\im(FM_1\cdots M_{k-i}) = H_i \text{ for all $0\leq i\leq k$}),
    \end{multline}
    so to prove \eqref{eq:chainspan_wts} it suffices to prove
    \begin{equation}\label{eq:moment-formula}
    \lim_{n \to \infty} \sum_{F\in \Hom_{R}(R^n,G)} \P_{M_1,\dots, M_k}(\im(FM_1\cdots M_{k-i}) = H_i \text{ for }0\leq i\leq k) = 1.
\end{equation}
We will massage this expression still further to match the one treated in \Cref{thm:chain:span}. For any $R$-module $G$ and $F\in \Hom_{R}(R^n,G)$, we may canonically identify $F$ with a \emph{row} vector
\begin{equation}
    F_{\mathrm{row}}=[F(\mathbf{v}_1) \; \ldots \; F(\mathbf{v}_n)]\in G^{1\times n},
\end{equation}
where $\mathbf{v}_1,\dots,\mathbf{v}_n$ is the standard basis of $R^n$. As usual, the row vector corresponding to $FM_1\cdots M_i$ is given by the matrix multiplication $F_{\mathrm{row}}M_1\cdots M_i$, and the image of $F$ is the column space of $F_{\mathrm{row}}$, which is nothing but the span of each entry of $F_{\mathrm{row}}$ (note that each column is just an entry since $F_{\mathrm{row}}$ has only one row). We use the notation $\sp(\,\cdot\,)$ for any row vector or column vector with entries in $G$ to denote the $R$-span of its entries in $G$. Thus to show \Cref{thm:moment} it suffices to show
\begin{equation}\label{eq:limit_rowvector}
    \lim_{n\to \infty}\sum_{F\in G^{1\times n}} \Prob(\sp(FM_1\cdots M_{k-i})=H_i \text{ for all $0\leq i\leq k$})=1.
\end{equation}
Now, we replace every matrix by its transpose, and note that if $M_1,\dots,M_k$ are independent and $\alpha$-balanced, then $M_1^T,\dots,M_k^T$ are also independent and $\alpha$-balanced, with potentially different measures (because the entries may not be identically distributed). Let $G^n=G^{n\times 1}$ denote column $n$-vectors with entries in $G$. Hence by setting $N_i=M_i^T$, \eqref{eq:limit_rowvector} is equivalent to
    \begin{equation}\label{eq:pre_reindexing}
        \lim_{n\to \infty}\sum_{F\in G^n} \Prob(\sp(N_{k-j}\cdots N_1F)=H_j \text{ for all $0\leq j\leq k$})=1.
    \end{equation}
    Make a cosmetic reindexing by setting $i=j+1$ for $0\leq j\leq k$ and abusing notation by writing $M_i=N_{k+1-i} \pmod{p^d}$ for $1\leq i\leq k$ (these are not the same $M_i$ as before). Then \eqref{eq:pre_reindexing} is equivalent to
    \begin{equation}\label{eq:almost_chain_span}
        \lim_{n\to\infty}\sum_{F\in G^n} \Prob(\sp(M_i\cdots M_k F)=H_{i-1} \text{ for all $1\leq i\leq k+1$})=1.
    \end{equation}
    Thus \Cref{thm:chain:span} implies \eqref{eq:almost_chain_span}, and by tracing back the above string of reductions, implies \eqref{eq:moment-formula} and hence \Cref{thm:moment} as well.
\end{proof}

\section{Lemmas for linear algebra over $\Z/a\Z$} \label{sec:supporting_lemmas}

Having gone to the trouble in the last section to change settings to that of finite rings, we will work in this setting for this and the next section. The purpose of this section is to introduce some supporting lemmas, mostly by Wood \cite{W1}, some of which were slightly extended in \cite[Section 2]{nguyen2022universality}. 

Throughout the next two sections fix $a \in \BBZ_{> 1}$ and set $R= \Z/a\Z$. Let $V=R^n$ with standard basis $\Bv_i,1 \leq i \leq n$. For $\sigma \subset [n]$ we denote by $V_{\sigma^c}$ the submodule generated by $\{\Bv_i: i \in \sigma^c\}$. Throughout the paper, to declutter notation we will write $(x_1,\ldots,x_n) \in R^n$ for usual (column) vectors, and similarly for vectors in e.g. $G^n$ where $G$ is a group, rather than using the notation $(x_1,\ldots,x_n)^T$. 

Sometimes it is convenient to identify $F \in \Hom(V,G)$ with the vector $(F(\Bv_1),\dots, F(\Bv_n)) \in G^n$, and we will usually abuse notation and view $F$ as a vector rather than a map. In particular, if $X = (x_1,\ldots,x_n) \in R^n$ is a vector, we write $\lang F, X \rang := \sum_{i=1}^n x_i F(\Bv_i)$; note this is not a usual dot product because $(F(\Bv_1),\dots, F(\Bv_n)) \in G^n$ and $(x_1,\ldots,x_n) \in R^n$ live in different spaces, though the formula is the same. If $M$ is an $n \times n$ matrix with entries in $R$, then for any $R$-module $G$, $M$ defines a linear map $G^n \to G^n$ by usual matrix multiplication, and we write $MF$ for the image of the vector $(F(\Bv_1),\dots, F(\Bv_n)) \in G^n$ under this map.


\subsection{Codes} 
\begin{definition} Given $w \le n$, we say that $F \in \Hom(V,G)$ is a code of distance $w$ if for every $\sigma \subset [n]$ with $\abs{\sigma} <w$ we have $F (V_{\sigma^c})=G$.
\end{definition}
Note that the number of codes can be approximated very precisely (see \cite[Claim 2.4]{nguyen2022universality})

\begin{claim}\label{claim:code:count} Let $G$ be a finite abelian group and let $\CC(G)$ be the number of codes (defined as $F(V)$) of distance $\delta n$ in $G$. We have 
$$\abs{\CC(G)} = (1+K'\exp(-c'_\delta n))\abs{G}^n,$$
where $K'$ depends on $G$ and $c_\delta$ depends on $\delta$.
\end{claim}

It is convenient to work with codes because the random walk $S_k = \sum_{i=1}^k x_i F(\Bv_i)$ (in discrete time indexed by $k=1,2,\ldots,n$) spreads out in $G$ very fast, as the following lemma shows. 

\begin{lemma}\cite[Lemma 2.1]{W1}\label{lemma:code:single:1} Assume that $x_i\in R$ are independent $\al$-balanced random variables. Then for any code $F$ of distance $\delta n$ and any $g\in G$,
$$\left|\P(\lang F, X \rang  =g) - \frac{1}{\abs{G}}\right| \le \exp(-\al \delta n/a^2),$$
where $X=(x_1,\dots,x_n)$.
\end{lemma}
In what follows, if not specified otherwise, $X$ is always understood as the random vector $(x_1,\dots, x_n)$ where $x_i$ are independent $\al$-balanced random variables as in Lemma \ref{lemma:code:single:1}.

Using the above result, it is not hard to deduce the following matrix form.
\begin{lemma}\cite[Lemma 2.4]{W1}\label{lemma:code:single:matrix:1} Assume that the entries of $M$ of size $n$ are independent $\al$-balanced random variables. For code $F$ of distance $\delta n$, for any vector $A \in G^n$ 
$$\left|\P(M F = A) - \frac{1}{\abs{G}^{n}}\right| \le \frac{K \exp(-cn)}{\abs{G}^{n}},$$
where $K,c$ depend on $a, G, \al$ and $\delta$.
\end{lemma}  

\begin{remark}
In our applications, $G$ will always be a fixed group (or perhaps summed over a finite collection of groups), so the dependence of the constants on $G$ which we allow in Lemma \ref{lemma:code:single:matrix:1} and similar results does not create any issue with our asymptotics.
\end{remark}

We will also need the following useful result (see also \cite[Lemma 2.3]{nguyen2022universality}), which can be deduced directly from Lemma \ref{lemma:code:single:matrix:1} above.
\begin{lemma}\label{lemma:code:subgp} Let $\delta$ be sufficiently small. Assume that $F\in Hom(V,G)$ is a code of distance $\delta n$. Assume that the entries of the matrix $M$ of size $n$ are iid copies of $\xi$ satisfying \eqref{eqn:alpha_R}. Then for any $H \le G$
$$\P(M F \mbox{ is a code of distance $\delta n$ in $H_{}$}) = \abs{H_{}}^{n} \frac{1+ O( \exp(-c''n))}{\abs{G}^n},$$
where $c''$ depends on $a,G,H, \delta, \al$.
\end{lemma}

\subsection{Non-codes} Next, for non-code $F$, the random walk $\lang F, X \rang $ does not converge quickly to the uniform distribution on $G$. However it is likely to be uniform over the subgroup where the restriction of $F$ is a code. 

\begin{definition}
For $D =\prod_i p_i^{e_i}$ let 
$$\ell(D):= \sum_i e_i.$$ 
\end{definition}

In all results introduced below we remark that $F$ is not necessarily a surjection. 

\begin{definition}\label{def:depth} For a real $\delta>0$, the $\delta$-depth of $F \in \Hom(V,G)$ is the maximal positive integer $D$ such that there exists $\sigma \subset [n]$ with $\abs{\sigma} < \ell(D) \delta n$ such that $D = \abs{G/F(V_{\sigma^c})}$, or is 1 if there is no such $D$.
 \end{definition}
So roughly speaking the $\delta$-depth measures the maximum of $\abs{G/F(V_{\sigma^c})}$ over $\sigma$ of size significantly smaller than $\delta n$. The depth is large if there exists such $\sigma$ where $F(V_{\sigma^c})$ is a small subgroup of $G$. The reason for this definition of depth is the following lemma, which shows that depth encodes how much one has to restrict $F$ to obtain a code.


\begin{lemma}\cite[Lemma 2.6]{W1}\label{lemma:non-code:count:1} 
The number of $F\in \Hom(V,G)$ with depth $D$ is at most
$$K \binom{n}{\lceil \ell(D) \delta n\rceil -1} \abs{G}^n D^{-n+\ell(D) \delta n},$$
where $K$ depends on $a$ and $G$. 
\end{lemma}

The following is \cite[Lemma 2.7]{nguyen2022universality}\footnote{ Although the result in \cite[Lemma 2.7]{nguyen2022universality} is stated for vectors $X$ with iid $\al$-balanced entries, the proof automatically extends to independent, $\al$-balanced ones that are not necessarily iid.}
\begin{lemma}\label{lemma:non-code:single:1} 
Let $F \in \Hom(V,G)$ have $\delta$-depth $D>1$ and $\abs{G/F(V)}<D$. Then for any $g\in G$
$$\P(\lang F, X \rang   = g) \le (1-\al) \left(\frac{D}{\abs{G}} + \exp(-\al \delta n/a^2)\right).$$
\end{lemma}
Using this result, we can obtain the following (see also \cite[Lemma 2.8]{nguyen2022universality}).
\begin{lemma}\label{lemma:non-code:1:matrix}
 If $F \in \Hom(V,G)$ has $\delta$-depth $D>1$ and $\abs{G/F(V)}<D$ as in the previous lemma, then for any $A \in G^n$,
$$\P(MF = A) \le K \exp(-\al n)\frac{D^n}{\abs{G}^n},$$
where $K$ depends on $a, G, \al$ and $\delta$.
\end{lemma}

\section{Moment computation: Proof of \Cref{thm:chain:span}} \label{sec:moment_computation}

Some of the ideas of the proof are motivated by \cite[Section 3]{nguyen2022universality}, although our method is more direct and significantly shorter.  

\subsection{Key inductive step.} Let $F_k=F \in G^n$ such that $\sp(F_k)=G=H_k$. More generally, for each $1\le i\le k+1$, let $F_{i-1} = M_i\dots M_k F_k$. In this notation, for \Cref{thm:chain:span} we are interested in the sum
$$\sum_{F_k\in H_k^n,\dots, F_1\in H_1^n, \sp(F_i) = H_i} \P(M_kF_k =F_{k-1}, M_{k-1} F_{k-1} = F_{k-2},\dots, M_2 F_2 =F_1, M_1F_1=0).$$
Let $F_k \in H_k^n$ be fixed. The following is key to our analysis.

\begin{proposition}\label{prop:single:k} With the same assumption as in \Cref{thm:chain:span}, the following holds for $\delta$ sufficiently small, and for some constants $c,K$ depending on $k,\al, G,a,\delta$.

 \begin{enumerate}[(i)]
\item  (Code) Assume that $F_k$ spans $G$ (i.e. $H_k$) and is a code of distance $\delta n$ in $G$. Then 
\begin{align*}
 \Big |\sum_{F_{k-1}\in H_{k-1}^n,\dots, F_1\in H_1^n, \sp(F_i) = H_i} & \P(M_kF_k =F_{k-1}, M_{k-1} F_{k-1} = F_{k-2},\dots, M_2 F_2 =F_1, M_1F_1=0) -\frac{1}{\abs{G}^n} \Big |\\
&\le  K \frac{\exp(-c n)}{\abs{G}^n}.
\end{align*}
\vskip .05in
\item (Non-code) Assume that $F_k$ spans $G$, and the $\delta$-depth of $F_k$ is $D_k\ge 2$. Then 
\begin{align*}
 \sum_{F_{k-1}\in H_{k-1}^n,\dots, F_1\in H_1^n, \sp(F_i) = H_i} & \P\Big(M_kF_k =F_{k-1}, M_{k-1} F_{k-1} = F_{k-2},\dots, M_2 F_2 =F_1, M_1F_1=0\Big) \\
& \le K \exp(-\al n/2)  \frac{D_k^n}{\abs{G}^n}.
\end{align*}
\end{enumerate}
\end{proposition}
Note that this result is somewhat similar to \cite[Proposition 3.1]{nguyen2022universality}, although the event under consideration is more local, that we are interested in the event $M_kF_k =F_{k-1}, M_{k-1} F_{k-1} = F_{k-2},\dots, M_2 F_2 =F_1, M_1F_1=0$ rather than just $M_1\dots M_k F_k=0$. In fact,  \cite[Proposition 3.1]{nguyen2022universality} follows from our current Proposition \ref{prop:single:k}. It is interesting (but not totally unexpected) that the main term $\frac{1}{\abs{G}^n}$ in the estimate for Code (case (i)) is independent of the choices of $H_1,\dots, H_{k-1}$.

\begin{proof}[Proof of Proposition \ref{prop:single:k}]
In what follows $K$ and $c$ may vary, and the implied constants in $O(.)$ are allowed to depend on $k,\al, G,a$ and $\delta$.  

We prove (i) and (ii) together by induction on $k$, assuming both (i) and (ii) hold for $k-1$ as the inductive hypothesis. When $k=1$, (i) and (ii) follow from Lemma \ref{lemma:code:single:matrix:1} and Lemma \ref{lemma:non-code:1:matrix} respectively. Next we consider $k\ge 2$. 

\vskip .1in

{\bf Codes.} We first prove (i) by working with $F_k$ a code of distance $\delta n$. 

We consider the event (in the $\sigma$-algebra generated by $M_k$) that $M_kF_k$ spans $H_{k-1}$ in two ways 
\begin{enumerate}[(1)]
\item $M_k F_k =F_{k-1}$, where $F_{k-1}$ is a given code of distance $\delta n$ in $H_{k-1}$; 
\vskip .05in
\item $M_k F_k=F_{k-1}$, where the given $F_{k-1}$ is not a code of distance $\delta n$, and hence has $\delta$-depth $D_{k-1}\ge 2$ in $H_{k-1}$.
\end{enumerate}
For the first case, we apply the induction hypothesis for (i) to obtain
\begin{align*}
 \Big |\sum_{F_{k-2}\in H_{k-2}^n,\dots, F_1\in H_1^n, \sp(F_i) = H_i} & \P(M_{k-1}F_{k-1} =F_{k-2}, M_{k-2} F_{k-2} = F_{k-3},\dots, M_2 F_2 =F_1, M_1F_1=0) -\frac{1}{|H_{k-1}|^n} \Big |\\
&\le  K \frac{\exp(-c n)}{|H_{k-1}|^n}.
\end{align*}
For the second case, we also apply the induction hypothesis for (ii) to obtain
\begin{align*}
 \sum_{F_{k-2}\in H_{k-2}^n,\dots, F_1\in H_1^n, \sp(F_i) = H_i} & \P(M_{k-1}F_{k-1} =F_{k-2}, M_{k-2} F_{k-2} = F_{k-3},\dots, M_2 F_2 =F_1, M_1F_1=0) \\
& \le K \exp(-\al n/2)  \frac{D_{k-1}^n}{|H_{k-1}|^n}.
\end{align*}
 Hence
\begin{align*}
 &\sum_{F_{k-1}\in H_{k-1}^n,\dots, F_1\in H_1^n, \sp(F_i) = H_i}  \P(M_{k}F_{k} =F_{k-1}, M_{k-1} F_{k-1} = F_{k-2},\dots, M_2 F_2 =F_1, M_1F_1=0)\\
 &= \sum_{\text{$F_{k-1}$ is $\delta n$-code in $H_{k-1}$}} \P(M_k F_k=F_{k-1}) \sum_{\substack{F_{k-2}\in H_{k-2}^n,\dots, F_1\in H_1^n,\\ \sp(F_i) = H_i}}  \P(M_{k-1}F_{k-1} =F_{k-2},\dots, M_1F_1=0)\\
 &+ \sum_{\substack{D_{k-1} \geq 2 \\ D_{k-1} \big{\vert} |H_{k-1}|}}\sum_{\substack{F_{k-1}  \text{has $\delta$-depth $D_{k-1}$ in $H_{k-1}$},\\ \sp(F_{k-1})=H_{k-1}}} \sum_{\substack{F_{k-2}\in H_{k-2}^n,\dots, F_1\in H_1^n,\\ \sp(F_i) = H_i}} \P(M_k F_k=F_{k-1}) \P(M_{k-1}F_{k-1} =F_{k-2},\dots, M_1F_1=0)\\
  &=: S_1(H_{k-1}) + \sum_{\substack{D_{k-1} \geq 2 \\ D_{k-1} \big{\vert} |H_{k-1}|}}S_2(H_{k-1}, D_{k-1}).
 \end{align*}

For the first sum, by Claim \ref{claim:code:count}, and then by Lemma \ref{lemma:code:single:matrix:1} and the inductive hypothesis for (i) we have
\begin{align*}
S_1(H_{k-1})&=  |\CC(H_{k-1})| \frac{1+ K \exp(-c n)}{\abs{G}^n} \left(\frac{1}{|H_{k-1}|^n} + O\left(\frac{\exp(-c n)}{|H_{k-1}|^n}\right)\right) \\
&=  \frac{1}{\abs{G}^n} +  O\left(\frac{\exp(-cn)}{\abs{G}^n}\right).
\end{align*}
For the second sum, for each $D_{k-1}$ we apply Lemma \ref{lemma:non-code:count:1} and Lemma \ref{lemma:code:single:matrix:1} and the inductive hypothesis for (ii) to bound 
\begin{align*}
S_2(H_{k-1},D_{k-1})& \le  K' \binom{n}{\lceil \ell(D_{k-1}) \delta n\rceil -1} |H_{k-1}|^n D_{k-1}^{-n+\ell(D_{k-1}) \delta n}  \frac{1+ K'' \exp(-c n)}{\abs{G}^n} \times  K \exp(-\al n/2)\frac{D_{k-1}^n}{|H_{k-1}|^n} \\
&=O\left( \frac{ \exp(-\al n/4) }{\abs{G}^n}\right),
\end{align*}
where for the second line we recall that $\delta$ was chosen sufficiently small and $n$ is sufficiently large. %

Summing over divisors $D_{k-1}$ of $|H_{k-1}|$,
\begin{align*}
\sum_{D_{k-1} \geq 2, D_{k-1}\big{\vert} \left\vert H_{k-1}\right\vert} S_2(H_{k-1},D_{k-1}) =O\left( \frac{ \exp(-\al n/4) }{\abs{G}^n}\right).
\end{align*} 
Summing over we thus obtain
\begin{align}\label{eqn:code:M_k:1}
\begin{split}
& \sum_{H_{k-1}} S_1(H_{k-1}) + \sum_{\substack{D_{k-1} \geq 2 \\ D_{k-1} \big{\vert} |H_{k-1}|}}S_2(H_{k-1}, D_{k-1})  \\
& = \frac{1}{\abs{G}^n}+O\left(\frac{\exp(-cn)}{\abs{G}^n}\right) + O\left( \frac{ \exp(-\al n/4) }{\abs{G}^n}\right),
\end{split}
\end{align}
completing the estimates for codes.
\vskip .1in
{\bf Non-codes.} We next prove (ii) by working with $F_k$ of $\delta$-depth $D_k\ge 2$, where $D_k$ also divides $|H_k|=\abs{G}$. Similarly to the previous part, we again compute the probability that $M_kF_k$ spans $H_{k-1}$ in the two possible ways:

\begin{enumerate}[(1)]
\item $M_k F_k=F_{k-1}$, where $F_{k-1}$ is a given code of distance $\delta n$ in $H_{k-1}$; 
\vskip .05in
\item $M_k F_k=F_{k-1}$, where the given $F_{k-1}$ is not a code of distance $\delta n$, and hence has $\delta$-depth $D_{k-1}\ge 2$ in $H_{k-1}$.
\end{enumerate}

For the first case, by Lemma \ref{lemma:non-code:1:matrix}, the probability with respect to $M_k$ is bounded by $K \exp(-\al n)  \frac{D_{k}^n}{\abs{G}^n}$.

Hence, by induction and by the independence of $M_1,\dots, M_k$

\begin{align*}
 &\sum_{\substack{F_{k-1}\in H_{k-1}^n\\ \text{$F_{k-1}$ is code of distance $\delta n$}}} \P(M_k F_k =F_{k-1})   \sum_{F_{k-2} \in H_{k-2}^n,\dots, F_1\in H_1^n, \sp(F_i) = H_i}  \P(M_{k-1}F_{k-1} =F_{k-2},\dots, M_1F_1=0)\\
 &= |\CC(H_{k-1})|K \exp(-\al n)  \frac{D_{k}^n}{\abs{G}^n} (\frac{1+K \exp(-c n)}{|H_{k-1}|^n}) \\
 &= O\left(\exp(-\al n/2) \frac{D_{k}^n}{\abs{G}^n}\right).
 \end{align*}

For the second case (2), the probability with respect to $M_k$, by Lemma \ref{lemma:non-code:count:1} and Lemma  \ref{lemma:non-code:1:matrix}, is bounded by
$$\sum_{\substack{F_{k-1}\in H_{k-1}^n\\ \text{$F_{k-1}$ has $\delta$-depth $D_{k-1}$}}} \P(M_k F_k =F_{k-1})   \le K\binom{n}{\lceil \ell(D_{k-1}) \delta n\rceil -1} |H_{k-1}|^n D_{k-1}^{-n+\ell(D_{k-1}) \delta n}  \times  K'\exp(-\al n)  \frac{D_{k}^n}{\abs{G}^n} .$$
Hence, by induction (applied to $M_1 \cdots M_{k-1}$ with the starting vector $F_{k-1}$)
\begin{align*}
& \sum_{\substack{F_{k-1}\in H_{k-1}^n\\ \text{$F_{k-1}$ has $\delta$-depth $D_{k-1}$}}} \P(M_k F_k =F_{k-1})  \sum_{F_{k-2} \in H_{k-2}^n,\dots, F_1\in H_1^n, \sp(F_i) = H_i}  \P(M_{k-1}F_{k-1} =F_{k-2},\dots, M_1F_1=0)\\
 &\le  K\binom{n}{\lceil \ell(D_{k-1}) \delta n\rceil -1} |H_{k-1}|^n D_{k-1}^{-n+\ell(D_{k-1}) \delta n}  \times  K'\exp(-\al n)  \frac{D_{k}^n}{\abs{G}^n} \times  K \exp(-\al n/2)  \frac{D_{k-1}^n}{|H_{k-1}|^n}\\
& =O\left( \exp(-\al n/2)\frac{D_{k}^n}{\abs{G}^n}\right),
\end{align*}
provided that $\delta$ was chosen sufficiently small and $n$ is sufficiently large.
\end{proof}

\begin{proof}[Proof of \Cref{thm:chain:span}]
We have
\begin{align*}
&\sum_{F_k\in H_k^n,\dots, F_1\in H_1^n, \sp(F_i) = H_i} \P(M_kF_k =F_{k-1},\dots, M_1F_1=0)\\
& = \sum_{\substack{F_{k}\in H_{k}^n\\ \text{$F_{k}$ is code of distance $\delta n$}}}   \sum_{F_{k-1}\in H_{k-1}^n,\dots, F_1\in H_1^n, \sp(F_i) = H_i}  \P(M_{k}F_{k} =F_{k-1},\dots, M_1F_1=0)\\
&+ \sum_{\substack{D_{k} \geq 2 \\ D_{k} \big{\vert} |H_{k}|}}  \sum_{\substack{F_{k}\in H_{k}^n\\ \text{$F_{k}$ has $\delta$-depth $D_{k}$}}}  \sum_{F_{k-1}\in H_{k-1}^n,\dots, F_1\in H_1^n, \sp(F_i) = H_i}  \P(M_{k}F_{k} =F_{k-1},\dots, M_1F_1=0)\\
&=S_1 + S_2.
\end{align*}
The first sum, by Claim \ref{claim:code:count} and by (i) of Proposition \ref{prop:single:k}, can be estimated as
$$S_1=  |\CC(G)| \frac{1}{\abs{G}^n} (1+ K \exp(-c n)) = 1 + O(\exp - cn) + O(\exp - c_\delta' n).$$
The second sum, by Lemma \ref{lemma:non-code:count:1} and by (i) of Proposition \ref{prop:single:k}, can be estimated as
\begin{align*}
S_2 &=  \sum_{\substack{D_{k} \geq 2 \\ D_{k} \big{\vert} \abs{G}}} K\binom{n}{\lceil \ell(D_{k}) \delta n\rceil -1} \abs{G}^n D_{k}^{-n+\ell(D_{k}) \delta n}  \times  K \exp(-\al n/2)  \frac{D_k^n}{\abs{G}^n}\\
&= O(\exp(-\al n/4)),
\end{align*}
provided that $\delta$ was chosen sufficiently small and $n$ is sufficiently large.
\end{proof}

\begin{remark}\label{rmk:largek} By using more explicit bounds and keeping track of the errors more carefully after each step of induction (see for instance \cite[Sections 3 and 4]{nguyen2024rank}, where the concrete problem is different, but the overall method seems applicable), the conclusion 
\begin{equation}\label{eqn:largek}
\sum_{F\in G^n} \Prob_{M_1,\dots,M_k\in \Mat_n(\Zp)}\Big(\sp(M_i\cdots M_k F)=H_{i-1} \text{ for all $1\leq i\leq k+1$}\Big)=1+o_n(1)
\end{equation}
of \Cref{thm:chain:span} seems to hold even when $k\le n^{c}$ and $\abs{G} = |H_k| \le e^{n^c}$ for some sufficiently small constant $c$. However, within the setting of $k\to \infty$ with $n$, while \eqref{eqn:largek} seems surprising, Theorem \ref{thm:flag} is rather vacuous.  
\end{remark}

\begin{remark}\label{rmk:smallalpha}
For fixed \( k \), by using the more detailed bounds from \cite[Section 4]{nguyen2022random} instead of those from Section 5 above, it seems possible that our results (Propositions \ref{thm:chain:span} and \ref{prop:single:k}) can be extended to the case where \( \alpha \geq \alpha_n = \Delta \frac{\log n}{n} \) for sufficiently large \( \Delta \) (independently of $p$). However, this level of sparsity is not our main focus.
\end{remark}

\section{Verifying well-behavedness} \label{sec:verify_well_behaved}
Here we prove Proposition \ref{prop:well-behaved}. We first apply \cite[Lemma~6.23]{sawin2022moment} to reduce it to a well-behavedness statement on the category of finite abelian $p$-groups.

\begin{lemma}\label{lem:moment-passing}
    For an abelian $p$-group $G$, let $n_k(G)$ be the number of injective $k$-flags in $G$, c.f.~\cite[Def.~6]{nguyen2022universality}. Then the moment function $\boldsymbol{1}$ is well-behaved on $\CC$ if and only if $G\mapsto n_k(G)$ is well-behaved on $\FinMod_p$.
\end{lemma}
\begin{proof}
    Recall that $\CC=(\FinMod_p,\calFl_k)$. By \cite[Lemma~6.23]{sawin2022moment}, $\boldsymbol{1}$ is well-behaved on $\CC$ if and only if the moment function
    \begin{equation}
        G\mapsto \sum_{\CF\in \calFl_k(G)} \boldsymbol{1}_{(G,\CF)}
    \end{equation}
    is well-behaved on $\FinMod_p$. But the sum is precisely $\abs{\calFl_k(G)}=n_k(G)$.
\end{proof}

We will use an upper bound of $n_k(G)$ from \cite{nguyen2022universality}. Given an integer partition $\lambda=(\lambda_1\geq \dots \geq \lambda_l)$, define 
\begin{equation}
    G_\lambda = \frac{\Z}{p^{\lambda_1}\Z} \oplus \cdots \frac{\Z}{p^{\lambda_l}\Z}.
\end{equation}
By the classification of modules over PID, for each $G\in \FinMod_p$, there is a unique partition $\lambda$, called the \defn{type} of $G$, such that $G\simeq G_\lambda$. 

\begin{lemma}[{\cite[Lemma~8.5]{nguyen2022universality}}]
    Given a prime $p$ and an integer $k\geq 1$, there are constants $f_k\in \R$ and $0<c_k<1$ depending only on $k$, such that
    \begin{equation}\label{eq:nk-bound}
        n_k(G_\lambda)\leq p^{f_k\lambda_1+\frac{1-c_k}{2}\sum_{i\geq 1}\lambda_i'^2},
    \end{equation}
    for any partition $\lambda$, where $\lambda_i':=\abs{\{j: \la_j \geq i\}}$ is the $i$-part of the conjugate partition of $\lambda$. 
\end{lemma}

We will use this upper bound to show that $n_k(\cdot)$ is well-behaved on $\FinMod_p$. Unfortunately, neither of the two general-purpose results \cite[Corollary~6.5]{sawin2022moment} nor \cite[Lemma~6.9]{sawin2022moment} applies directly: \cite[Corollary~6.5]{sawin2022moment} does not allow the ``quadratic growth'' $\sum \lambda_i'^2$ in the exponent, while \cite[Lemma~6.9]{sawin2022moment} does not allow the extra ``linear growth'' $f_k\lambda_1$ in the exponent. To this end, we will prove a slightly stronger version of \cite[Lemma~6.9]{sawin2022moment} for our current setting as Lemma \ref{lem:quadratic-saving} below. For a quick glance, it essentially says that the constraint on the coefficient of the linear growth is only relevant when the coefficient of the quadratic growth is the maximal allowed (namely, the quadratic part grows like $\sum \lambda_i'^2/2$). We will state Lemma \ref{lem:quadratic-saving} over $\Zp$ for simplicity, but the same holds if $\Zp$ is replaced by any discrete valuation ring with finite residue field.

For $G\in \FinMod_p$, let $\wedge^2 G$ denote the quotient of the $\Zp$-module $G\otimes G$ by the submodule generated by elements of the form $g\otimes g$ for $g\in G$. Recall for any partition $\lambda$ that $\abs{G_\lambda}=p^{\sum_{i\geq 1}\lambda_i'}$ and (see e.g. \cite[p.~925]{W0})
\begin{equation}\label{eq:alt2-size}
    \abs{\wedge^2 G_\lambda}=p^{\sum_{i\geq 1} \frac{\lambda_i'(\lambda_i'-1)}{2}}.
\end{equation}
For $s\in \Z_{\geq 1}$, let $\FinMod_{\Z/p^s\Z}$ be the full subcategory of $\FinMod_p$ consisting of all $\Z/p^s\Z$-modules, or equivalently, all $G_\lambda$ with $\lambda_1\leq s$. 

\begin{lemma}\label{lem:levelwise}
    Let $N=(N_G)_G\in \R^{\FinMod_p/\simeq}$ be a moment function on $\FinMod_p$. Suppose for each $s\in \Z_{\geq 1}$, there exist $\eps,c>0$ (possibly depending on $s$) such that $\abs{N_G}\leq c\abs{G}^{1-\eps}\abs{\wedge^2G}$ for all $G\in \FinMod_{\Z/p^s \Z}$. Then $N$ is well-behaved on $\FinMod_p$.
\end{lemma}
\begin{proof}
    As is in the proof of \cite[Lemma~6.9]{sawin2022moment}, it suffices to prove for each fixed $G\in \FinMod_p$ that
    \begin{equation}\label{eq:moment-sum}
        \frac{1}{\abs{\Aut(G)}}\sum_{e=0}^\infty \parens*{\frac{1}{\abs{\Hom(G,\Fp)}^e\prod_{j=1}^e (p^j-1)}}\sum_{\alpha\in \mathrm{Ext}^1(N,(\Fp)^e)} \abs{N_{G_\alpha}} < \infty,
    \end{equation}
    where $0\to G\to G_\alpha \to (\Fp)^e$ is the extension determined by $\alpha$. Now let $s=\lambda_1+1$, where $\lambda$ is the type of $G$. Then the extension sequence implies that $G_\alpha\in \FinMod_{\Z/p^s \Z}$ for every $\alpha$. Choose $\eps,c>0$ based on this $s$, then the sum in \eqref{eq:moment-sum} is bounded above by
    \begin{equation}
        \frac{1}{\abs{\Aut(G)}}\sum_{e=0}^\infty \parens*{\frac{1}{\abs{\Hom(G,\Fp)}^e\prod_{j=1}^e (p^j-1)}}\sum_{\alpha\in \mathrm{Ext}^1(N,(\Fp)^e)} c\abs{G}^{1-\eps}\abs{\wedge^2G},
    \end{equation}
    which is shown to be finite in the proof of \cite[Lemma~6.9]{sawin2022moment}.
\end{proof}

\begin{lemma}\label{lem:quadratic-saving}
    Let $N=(N_G)_G\in \R^{\FinMod_p/\simeq}$ be a moment function on $\FinMod_p$. Suppose there exist $\eps,r,c>0$ such that $\abs{N_G}\leq c\abs{G}^r\abs{\wedge^2G}^{1-\eps}$ for all $G\in \FinMod_p$. Then $N$ is well-behaved on $\FinMod_p$. 
\end{lemma}
\begin{proof}
    Fix an arbitrary $s\in \Z_{\geq 1}$ and consider all partitions $\lambda$ with $\lambda_1\leq s$, so that $G_\lambda\in \FinMod_{\Z/p^s \Z}$ and $\lambda_{i}'=0$ for $i>s$. Cauchy--Schwarz inequality implies that
    \begin{equation}
        \sum_{i=1}^s \lambda_i' \leq \sqrt{s}\parens*{\sum_{i=1}^s \lambda_i'^2}^{1/2}.
    \end{equation}
    Since there are only finitely many integer-valued vectors $(\lambda_i')_{i=1}^s$ with bounded $\sum_{i=1}^s \lambda_i'^2$, the above estimate implies that for every $f\in \R$, there exists $C$ depending on $s$ and $f$ such that
    \begin{equation}
        \sum_{i=1}^s \lambda_i'^2 - f \sum_{i=1}^s \lambda_i' \geq -C
    \end{equation}
    for all partitions $\lambda$ with $\lambda_1\leq s$. Combined with \eqref{eq:alt2-size}, we may absorb $\abs{G}^r$ into $\abs{\wedge^2 G}^{-\eps}$ to conclude that there exists $c'$ depending on $s$ such that
    \begin{equation}
        \abs{N_G}\leq c' \abs{\wedge^2 G}
    \end{equation}
    for all $G\in \FinMod_{\Z/p^s \Z}$. Since $s\in \Z_{\geq 1}$ is arbitrary, the conclusion follows from Lemma \ref{lem:levelwise}.
\end{proof}

\begin{remark}~
    \begin{enumerate}
        \item By \cite[Lemma~6.9]{sawin2022moment} and the remark thereafter, $N_G=\abs{G}^r\abs{\wedge^2 G}$ is not well-behaved on $\FinMod_p$ if and only if $r\geq 1$. Note that this does not contradict Lemma \ref{lem:quadratic-saving}.
     \vskip .05in
   \item While \cite[Corollary~6.5]{sawin2022moment} is not a special case of \cite[Lemma~6.9]{sawin2022moment}, it is now a special case of both Lemma \ref{lem:levelwise} and Lemma \ref{lem:quadratic-saving}.
    \end{enumerate}
\end{remark}

\begin{proof}[Proof of Proposition \ref{prop:well-behaved}]
    In \eqref{eq:nk-bound}, note that $\lambda_1\leq \sum_{i\geq 1} \lambda_i'$ and the coefficient $\frac{1-c_k}{2}$ in front ot $\sum_{i\geq 1}\lambda_i'^2$ is strictly less than $1/2$. Comparing this coefficient with \eqref{eq:alt2-size}, bounding the factors involving $p^{\sum \lambda_i'}$ from above by $\abs{G}^r$ for some large $r$, and applying Lemma \ref{lem:quadratic-saving}, it follows that $G\mapsto n_k(G)$ is well-behaved on $\FinMod_p$. By Lemma \ref{lem:moment-passing}, we conclude that $\boldsymbol{1}$ is well-behaved on $\CC$.
\end{proof}

\section{Proof of two-matrix corollaries} \label{sec:prove_conv_cor}



\begin{proof}[Proof of {\Cref{thm:conv_cor_intro}}]
        We will deduce from the $k=2$ case of \Cref{thm:flag} and some calculation. We relate the desired probability to probabilities with no conditioning via
     \begin{equation}\label{eq:cond_prob_mm}
         \text{LHS\eqref{eq:conv_cor_intro}} = \frac{\lim_{n \to \infty} \Pr\Big(\cok(M_1 M_2)_P \simeq G \text{ and } \cok(M_1)_P \simeq H \text{ and }\cok(M_2)_P \simeq K\Big)  }{\lim_{n \to \infty} \Pr(\cok(M_1)_P \simeq H \text{ and }\cok(M_2)_P \simeq K)}.
     \end{equation}
     The denominator is 
    \begin{equation}\label{eq:authk}
        \lim_{n \to \infty} \Pr(\cok(M_1)_P \simeq H \text{ and }\cok(M_2)_P \simeq K) = \frac{\left(\prod_{p \in P} \prod_{i \geq 1} (1-p^{-i})\right)^2}{\abs{\Aut(H)} \cdot \abs{\Aut(K)}}
    \end{equation}
    by applying the $k=1$ case of \Cref{thm:flag} (which, as mentioned, was proven earlier in \cite{W1}) to the independent matrices $M_1$ and $M_2$. 

    Now we must compute the numerator in \eqref{eq:cond_prob_mm}, which is the slightly harder part. Recall the notation $\Z_{(P)}$ from \eqref{def:ZP}, and note that matrices over $\Z$ trivially act on $\Z_{(P)}^n$ since $\Z$ includes into $\Z_{(P)}$. Let $M_1,M_2$ be any deterministic matrices for the moment. If $\phi:\cok(M_1M_2)_P \to \cok(M_1)_P$ is the natural map, then 
    \begin{equation}
         \ker(\phi) =  M_1(M_2 \Z_{(P)}^n)/M_1 \Z_{(P)}^n.
     \end{equation}
     If $M_1$ is full rank over $\Q$, then multiplication by $M_1^{-1}$ is an isomorphism of $\Z_{(P)}$-modules, and hence
     \begin{equation}
         \ker(\phi) \simeq M_2 \Z_{(P)}^n / \Z_{(P)}^n = \cok(M_2)_P.
     \end{equation}
     Since we are conditioning on $\cok(M_1)_P$ being a finite group in \eqref{eq:conv_cor_intro}, we are in particular conditioning on $M_1$ being full rank.
     Hence
     \begin{multline}\label{eq:cok_to_flags}
         \Prob\Big(\cok(M_1 M_2)_P \simeq G \text{ and } \cok(M_1)_P \simeq H \text{ and }\cok(M_2)_P \simeq K\Big) \\ 
         = \sum_{\substack{[G_2 \xrightarrow{\phi} G_1] \in \Fl_{2,P}/\simeq \text{ with} \\ \ker(\phi) \simeq K, G_2 \simeq G, G_1 \simeq H}} \Prob(\bfcok(M_1,M_2) \simeq (G_2 \xrightarrow{\phi} G_1)).
     \end{multline}
    Applying \Cref{thm:flag},
     \begin{equation}\label{eq:flags_and_auts}
         \text{RHS\eqref{eq:cok_to_flags}} = \left(\prod_{p \in P} \prod_{i \geq 1} (1-p^{-i})\right)^2 \sum_{\substack{[G_2 \xrightarrow{\phi} G_1] \in \Fl_{2,P}/\simeq \text{ with} \\ \ker(\phi) \simeq K, G_2 \simeq G, G_1 \simeq H}} \frac{1}{\abs{\Aut(G_2 \xrightarrow{\phi} G_1)}}.
     \end{equation}
     Now, $\Aut(G) \times \Aut(H)$ acts on $\Fl_2$ via 
     \begin{equation}
         (\gamma,\eta) \cdot (G_2 \xrightarrow{\phi} G_1) = (G_2 \xrightarrow{\eta \circ \phi \circ \gamma^{-1}} G_1),
     \end{equation}
     i.e. it sends the flag $G_2 \xrightarrow{\phi} G_1$ to the unique flag $G_2 \xrightarrow{\phi'} G_1$ such that the diagram

\[\begin{tikzcd}
    G & H \\
    G & H
    \arrow["\phi", from=1-1, to=1-2]
    \arrow["\gamma", from=1-1, to=2-1]
    \arrow["\eta", from=1-2, to=2-2]
    \arrow["{\phi'}", from=2-1, to=2-2]
\end{tikzcd}\]
    commutes. The stabilizer of a given element $(G_2 \xrightarrow{\phi} G_1) \in \Fl_2$ is exactly $\Aut(G_2 \xrightarrow{\phi} G_1)$. Hence by the orbit-stabilizer theorem,
    \begin{align}\label{eq:orb-stat-1}
        \begin{split}
            \text{RHS\eqref{eq:flags_and_auts}} &= \frac{\left(\prod_{p \in P} \prod_{i \geq 1} (1-p^{-i})\right)^2}{\abs{\Aut(G) \times \Aut(H)}}\sum_{\substack{[G_2 \xrightarrow{\phi} G_1] \in \Fl_{2,P}/\simeq \text{ with} \\ \ker(\phi) \simeq K, G_2 \simeq G, G_1 \simeq H}} \frac{\abs{\Aut(G) \times \Aut(H)}}{\abs{\Aut(G_2 \xrightarrow{\phi} G_1)}} \\ 
            &=   \frac{\left(\prod_{p \in P} \prod_{i \geq 1} (1-p^{-i})\right)^2}{\abs{\Aut(G) \times \Aut(H)}}\sum_{\substack{[G_2 \xrightarrow{\phi} G_1] \in \Fl_{2,P}/\simeq \text{ with} \\ \ker(\phi) \simeq K, G_2 \simeq G, G_1 \simeq H}} \abs{\{\psi: G \surj H: (G \xrightarrow{\psi} H) \simeq (G_2 \xrightarrow{\phi} G_1)\}} \\ 
            &= \frac{\left(\prod_{p \in P} \prod_{i \geq 1} (1-p^{-i})\right)^2}{\abs{\Aut(G) \times \Aut(H)}}\sum_{\substack{\psi: G \surj H \text{ with } \\ \ker(\psi) \simeq K}} 1.
        \end{split}
    \end{align}
    For each subgroup $N \leq G$ with $N \simeq K$ and $G/N \simeq H$, perhaps multiple $\psi$ will have $\ker(\psi) = N$. How many? Each $\psi$ is given by a composition of the quotient map $G \to G/N$ with an isomorphism $G/N \to H$, and there are $\abs{\Aut(H)}$ such isomorphisms, hence
    \begin{equation}
        \text{RHS\eqref{eq:orb-stat-1}} = \left(\prod_{p \in P} \prod_{i \geq 1} (1-p^{-i})\right)^2 \frac{\abs{\{N \leq G: N \simeq K, G/N \simeq H\}}}{\abs{\Aut(G)}}.
    \end{equation}
    Tracing back the string of equalities, 
    \begin{multline}
        \lim_{n \to \infty} \Pr(\cok(M_1 M_2)_P \simeq G \text{ and } \cok(M_1)_P \simeq H \text{ and }\cok(M_2)_P \simeq K) \\ 
        = \left(\prod_{p \in P} \prod_{i \geq 1} (1-p^{-i})\right)^2 \frac{\abs{\{N \leq G: N \simeq K, G/N \simeq H\}}}{\abs{\Aut(G)}}.
    \end{multline}
    By \eqref{eq:cond_prob_mm}, dividing the above by \eqref{eq:authk} yields the answer, and indeed doing so yields the right hand side of \eqref{eq:conv_cor_intro}.
\end{proof}

\begin{proof}[Proof of {\Cref{thm:rank_conv_intro}}]
    By reducing the matrices in the $P=\{p\}$ case of \Cref{thm:conv_cor_intro} modulo $p$, and summing over all abelian $p$ groups $G, H,K$ with ranks $c,a,b$, we obtain that 
    \begin{equation}
        \lim_{n \to \infty} \Prob\Big(\corank(M_1M_2) = c\mid \corank(M_1)=a \text{ and }\corank(M_2) = b\Big) 
    \end{equation}
    exists and is independent of the choice of matrix distribution. It is given by a sum over $G,H,K$, but rather than compute this sum we will simply notice that it suffices to compute the limit directly in the case where $M_1,M_2$ are uniformly distributed over $\Mat_n(\F_p)$. The uniform distribution is invariant under multiplication by $\GL_n(\F_p)$ on both the right and left, and there is a unique distribution supported on $\{M \in \Mat_n(\F_p) \mid \corank(M)=a\}$, given by $U\diag(I_{n-a},0_a)V$ where $U,V \in \GL_n(\F_p)$ are independent and uniform. Hence
    \begin{equation}
         M_1M_2 = U_1 \diag(I_{n-a},0_a)V_1 U_2 \diag(I_{n-b},0_b)V_2
     \end{equation} 
        in distribution where the $U_i$ and $V_i$ are independent in $\GL_n(\F_p)$. The corank is independent of $U_1$ and $V_2$, and $V_1U_2$ has uniform distribution, so 
        \begin{multline}
                    \Prob\Big(\corank(M_1M_2) = c\mid \corank(M_1)=a \text{ and }\corank(M_2) = b\Big) \\= \Prob\Big(\corank(\diag(I_{n-a},0_a) U \diag(I_{n-b},0_b)) = c\Big)
        \end{multline}
    where $U \in \GL_n(\F_p)$ is uniform. The matrix $\diag(I_{n-a},0_a) U \diag(I_{n-b},0_b)$ is just an $(n-a) \times (n-b)$ corner of a uniform matrix in $\GL_n(\F_p)$, padded with zeroes. Because the columns of a matrix in $\GL_n(\F_p)$ are uniform given linear independence, we have
    \begin{equation}
        \Prob\Big(\corank(M_1M_2) = c\mid \corank(M_1)=a \text{ and }\corank(M_2) = b\Big) = \Prob(\rank(B') = n-c)
    \end{equation}
    where $B'$ is the upper $(n-a) \times (n-b)$ submatrix of a uniformly full-rank matrix $B \in \Mat_{n \times (n-b)}$. Our reason for stating in this form is to use a known result, for which we also require the notation
    \begin{equation}
        \sqbinom{k}{\ell}_{p^{-1}} := \frac{(p^{-1};p^{-1})_k}{(p^{-1};p^{-1})_\ell (p^{-1};p^{-1})_{k-\ell}},
    \end{equation}
    where we recall $(p^{-1};p^{-1})_m$ was defined in the statement. Now, \cite[Lemma 6.7]{van2023reflecting}\footnote{One must take $n,d,k,r$ in the statement there to be $n-a,a,n-b,n-c$ in our notation. Also, one must assume that $a \geq b$ to satisfy the hypotheses there, but since the formula arrived at is symmetric in $a$ and $b$ (as it must be), the case $a < b$ follows by symmetry.} yields that 
    \begin{align}
        \begin{split}
            \Prob(\rank(B') = n-c) &= p^{-(c-a)(c-b)} \frac{\sqbinom{a}{c-b}_{p^{-1}} \sqbinom{n-a}{n-c}_{p^{-1}}}{\sqbinom{n}{n-b}_{p^{-1}}} \\ 
            &=   \frac{p^{-(c-a)(c-b)}(p^{-1};p^{-1})_a(p^{-1};p^{-1})_b}{(p^{-1};p^{-1})_{c-b}(p^{-1};p^{-1})_{a+b-c}(p^{-1};p^{-1})_{c-a}} \cdot \frac{(p^{-1};p^{-1})_{n-a}(p^{-1};p^{-1})_{n-b}}{ (p^{-1};p^{-1})_{n-c}(p^{-1};p^{-1})_n}.
        \end{split}
    \end{align}
    In the $n \to \infty$ limit, the $n$-dependent fraction becomes $1$, and we obtain the result.
\end{proof}

We note that the idea of the above argument is the same as the proof of \cite[Theorem 1.4]{nguyen2022universality}.

\appendix

\section{Symmetric polynomial formulas and parallelism}\label{section:formulas}

This Appendix elaborates on the discussion in the Introduction regarding structural connections with eigenvalues/singular values of complex matrix products. It is still relatively terse and impressionistic, but includes references to longer treatments.

\subsection{Additive and multiplicative free convolution} We first give a very brief overview of free convolutions, we refer the reader to \cite{AGZ,NS} for more details. Let $\mu$ and $\nu$ be probability measures on $\R$ with compact support. If $a,b$ are two freely independent random variables in a noncommutative probability space $(\CA,\phi)$ with law $\mu,\nu$ respectively, then the laws of $a+b$ and $ab$ are defined by $\mu\boxplus \nu$ and $\mu \boxtimes \nu$ respectively. We next give an analytical description of the convolutions. Consider the Cauchy transform of $\mu$,
$$G_{\mu}(z) = \int_{\R} \frac{d \mu(t)}{z-t}.$$
As $\mu$ has compact support, we can express it as formal series about $z=\infty$ 
 $$G_{\mu}(z) = \sum_{n\ge 0}m_{n} z^{-n-1} = \frac{1}{z}M_{\mu}(\frac{1}{z}),$$
 where $m_{n}$ are the moments of $\mu$ and $M_{\mu}(z)$ is the moment generating function.
 
Let $K_{\mu}(z)$ be the formal inverse of $G_{\mu}$, that is $G_{\mu}(K_{\mu}(z))=z$. Then $K_{\mu}(z)$ can be written as a formal series
$$K_{\mu}(z) = \frac{1}{z}+ \sum_{n=1}^{\infty} C_{n} z^{n}.$$
Define the $R$-transform as
$$R_{\mu}(z):= K_{\mu}(z) - 1/z.$$
Alternatively, we can define $R_{\mu}(z)$ as $\sum_{n=0}^{\infty} \kappa_{n+1}z^{n}$, where $\kappa_{n}$ are the free cumulants of $\mu$. Then $\mu\boxplus \nu$ is the measure whose $\R$-transform satisfies
$$R_{\mu \boxplus \nu}(z) = R_{\mu}(z) + R_{\nu}(z).$$
For the multiplicative free convolution, define the $S$-transform of $\mu$ by
$$S_{\mu}(z) := \frac{1+z}{z} M_{\mu}^{-1}(z),$$
where $M_{\mu}^{-1}(.)$ is the compositional inverse. Then $\mu\boxtimes \nu$ is the measure whose $S$-transform satisfies
$$S_{\mu \boxtimes \nu}(z) = S_{\mu}(z)  S_{\nu}(z).$$

\subsection{Symmetric polynomials and functions}

The treatment here follows \cite{GM} and \cite{van2020limits}. The Macdonald polynomials $P_\la (x_1,\dots, x_n;q,t)$ are symmetric polynomials associated to integer partitions $\la$ of length at most $n$, which is an $n$-tuple of non-negative integers $\la=(\la_{1},\dots, \la_{n}), \la_{1}\ge \dots \ge \la_{n}$. These are usually defined as eigenfunctions of some differential operator on the algebra $\Q(q,t)[x_1,\dots,x_n]^{S_n}$, and span the $\Q(q,t)$-vector space of such polynomials. More specifically, let $\Lambda_{n}$ be the algebra of symmetric polynomials in $n$ variables $x_{1},\dots, x_{n}$. Define a difference operator $D_{q,t}$ acting in $\Lambda_{n}$ as follows
$$D_{q,t} f = \sum_{i=1}^{n} \prod \frac{tx_{i} -x_{j}}{x_{i}-x_{j}} T_{q,x_{i}} f,$$
where $(T_{q,x_{i}}f)(x_{1},\dots, x_{n}) = f(x_{1},\dots, x_{i-1},qx_{i},x_{i+1},\dots, x_{n})$.

The operator $D_{q,t}$ has a complete set of eigenfunctions in $\Lambda_{n}$, which are the Macdonald polynomials $P_{\la}(x_{1},\dots, x_{n}; q,t)$, where $D_{q,t} P_{\la}(x_1,\ldots,x_n; q,t) = (\sum_{i=1}^{n} q^{\la_{i}} t^{n-i})P_{\la}(x_1,\ldots,x_n; q,t)$.
We also define the normalized polynomial
\begin{equation}
    \label{eq:mac_norm}
    \hat{P}_\la(x_1,\ldots,x_n;q,t) := \frac{P_\la(x_1,\ldots,x_n;q,t)}{P_\la(1,t,\ldots,t^{n-1};q,t)}.
\end{equation}

We list below a few special cases.

\subsubsection{Schur polynomials.} This is the Macdonald polynomial $P_\la (x_1,\dots, x_n;q,q)$ (that is $t=q$), which can be shown to have the following determinantal form independent of $q$:
$$P_{\la}(x_1,\ldots,x_n; q,q) = s_{\la}(x_1,\ldots,x_n) = \frac{\det [x_{i}^{\la_{j}+n-j}]_{i,j=1}^{n}}{\prod_{1\le i<j\le n}(x_{i}-x_{j})}.$$

\subsubsection{Hall-Littlewood polynomials} 
It can be shown that when setting $q=0$, one has explicit formulas:
$$P_{\la} (x_{1},\dots, x_{n};0, t) = \frac{1}{v_{\la}(t)} \sum_{\sigma \in S_{n}} \sigma\left(x_{1}^{\la_{1}}\dots x_{n}^{\la_{n}} \prod_{1\le i<j\le n} \frac{x_{i} -tx_{j}}{x_{i}-x_{j}}\right), \mbox{ where } v_{\la}(t) = \prod_{i\ge 0} \prod_{j=1}^{m_{i}(\la)} \frac{1-t^{j}}{1-t}.$$
This case is known as a \emph{Hall-Littlewood polynomial}, and relevant in combinatorics of abelian $p$-groups, $p$-adic groups, and $p$-adic random matrices (all of which are related).


\subsubsection{The Heckman–Opdam hypergeometric functions and multivariate Bessel functions} Let $\al_1> \dots >\al_n$ and $\theta>0$. Define the (type A) Heckman-Opdam hypergeometric function by
$$\CF_\al(z_1,\dots, z_n;\theta) := \lim_{\eps \to 0} \eps^{\theta n(n-1)/2}  P_{\lfloor \al_1/\eps\rfloor, \dots, \lfloor \al_n/\eps\rfloor} (e^{\eps z_1},\dots, e^{\eps z_n}; e^{-\eps}, e^{-\theta \eps}).$$
Let
$$\hat{\CF}_\al (z_1,\ldots,z_n;\theta):= \frac{\CF_\al(z_1,\ldots,z_n;\theta)}{\CF_\al(0, -\theta, \dots, (1-n)\theta;\theta)}.$$
A further limit
$$B_\al(z_1,\ldots,z_n;\theta) := \lim_{\eps \to 0} \CF_{\eps \alpha}(\eps^{-1}z_1,\ldots,\eps^{-1}z_n;\theta)$$
is referred to as the multivariate Bessel function.

\subsubsection{The coefficients} 
We can express $P_\la P_\mu$ as a linear combination of the $P_\nu$, 
$$P_\la P_\mu = \sum_\nu c_{\la,\mu}^\nu(q,t) P_\nu$$
for some coefficients $c_{\la,\mu}^\nu(q,t)$ (strictly speaking, in our definition we have fixed $n$ and these coefficients $c_{\la,\mu}^\nu(q,t) $ depend on $n$, but they stabilize and take the same value for all $n$ large enough). The coefficients $\hat{c}_{\la,\mu}^\nu(q,t)$ can be defined by $\hat{P}_\la \hat{P}_\mu = \sum_\nu \hat{c}_{\la,\mu}^\nu(q,t) \hat{P}_\nu$, or written explicitly as
\begin{equation}\label{eq:c_hat}
    \hat{c}_{\la,\mu}^\nu(q,t) = \frac{P_\nu(1,t,\ldots,t^{n-1};q,t)}{P_\la(1,t,\ldots,t^{n-1};q,t)P_\mu(1,t,\ldots,t^{n-1};q,t)}c_{\la,\mu}^\nu(q,t).
\end{equation}
These depend implicitly on $n$, and converge as $n \to \infty$ to finite limits, though they do not stabilize as do their un-normalized counterparts.

In parallel, for multivariate Bessel functions and Heckman–Opdam hypergeometric functions,
$$\hat{B}_\al(z_1,\ldots,z_n;\theta) \hat{B}_\beta(z_1,\ldots,z_n;\theta) = \int_\gamma b_{\al,\beta}^\gamma (\theta)\hat{B}_\gamma(z_1,\ldots,z_n;\theta) d \gamma_{Leb}$$
and
$$\hat{\CF}_\al(z_1,\ldots,z_n;\theta) \hat{\CF}_\beta(z_1,\ldots,z_n;\theta) = \int_\gamma c_{\al,\beta}^\gamma(\theta) \hat{\CF}_\gamma(z_1,\ldots,z_n;\theta) d \gamma_{Leb}.$$
These $b_{\al,\beta}^\gamma$ and $c_{\al,\beta}^\gamma$ are also limits of the Macdonald structure constants $c_{\la,\mu}^\nu(q,t)$.

\subsection{The convolutions, again}
The following was shown in \cite[Section 2]{GM}, which are the measures mentioned in Subsection \ref{subsection:consequences} for the eigenvalue and singular value profile.
\begin{theorem}\label{thm:ev_sv_convolution} Let $\al = (\alpha_1,\ldots,\alpha_n)$ and $\beta = (\beta_1,\ldots,\beta_n)$ be two decreasing sequences of real numbers, and $\delta$ denote the Dirac delta measure. Then
 $$d(\delta_\al \boxplus_2 \delta_\beta) =   b_{\al,\beta}^\gamma (1)d \gamma_{Leb},$$
and
 $$d(\delta_\al \boxtimes_2 \delta_\beta) =   c_{\al,\beta}^\gamma (1)d \gamma_{Leb}.$$
 \end{theorem}
The following was shown in \cite[Theorem 1.3 (3)]{van2020limits}, which is the universal measure $\mu\boxtimes_{HL} \nu$ mentioned in Section \ref{subsection:consequences}. Recall the notation $G_\la$ from \eqref{eq:G_la}.
\begin{theorem}\label{thm:VP} Let $M_1,M_2 \in \Mat_n(\Z_p)$ be independent matrices with independent entries distributed by the additive Haar measure on $\Z_p$. Then
$$\P\left(\cok(M_{1} M_{2}) \simeq G_\nu \mid \cok(M_{1})\simeq G_\la, \cok(M_{2}) \simeq G_\mu\right) =\hat{c}_{\la,\mu}^\nu(0,1/p) =: (\delta_{\la} \boxtimes_{HL,n} \delta_{\mu})(G_\nu).$$
\end{theorem}

In other words, both singular values/eigenvalues of complex matrices and cokernels of $p$-adic matrices are governed by convolution operations coming from two parallel degenerations of Macdonald polynomials.

\subsection{Asymptotic description of Hall-Littlewood convolution}

Recall that the normalized structure constants $\hat{c}_{\la,\mu}^\nu(q,t)$ implicitly depend on $n$ through \eqref{eq:c_hat}. By \Cref{thm:VP} one has
\begin{multline}
            \lim_{n \to \infty} \P\left(\cok(M_{1} M_{2}) \simeq G_\nu \mid \cok(M_{1})\simeq G_\la, \cok(M_{2}) \simeq G_\mu\right) \\ 
            = c_{\la,\mu}^\nu(0,1/p) \frac{P_\nu(1,p^{-1},\ldots,;0,p^{-1})}{P_\la(1,p^{-1},\ldots;0,p^{-1})P_\mu(1,p^{-1},\ldots;0,p^{-1})} 
        =: (\delta_{\la} \boxtimes_{HL} \delta_{\mu})(G_\nu). 
\end{multline}
By combining formulas to be found in \cite[Chapter III]{mac} one has a group-theoretic description
\begin{multline}
    c_{\la,\mu}^\nu(0,1/p) \frac{P_\nu(1,p^{-1},\ldots,;0,p^{-1})}{P_\la(1,p^{-1},\ldots;0,p^{-1})P_\mu(1,p^{-1},\ldots;0,p^{-1})} \\ = \frac{\abs{\Aut(G_\la)} \cdot \abs{\Aut(G_\mu)}}{\abs{\Aut(G_\nu)}} \abs{\{N \leq G_\nu: N \simeq G_\la, G/N \simeq G_\mu\}},
\end{multline}
where the polynomials in infinitely many variables $1,p^{-1},\ldots$ may be defined as limits of polynomials in finitely many variables but also have quite explicit formulas. Hence, the convolution operation we obtain in \Cref{thm:conv_cor_intro} is naturally defined in terms of Hall-Littlewood polynomials, and is (a limit of) an exact structural analogue of the complex eigenvalue/singular value version. The limit of the convolution operations of \Cref{thm:ev_sv_convolution} to free additive and multiplicative convolution is a bit more complicated, since it involves a scaling limit of $n$-atom measures to measures which have absolutely continuous parts.


\begin{thebibliography}{NVP24b}

\bibitem[AGZ]{AGZ}
Greg~W. Anderson, Alice Guionnet, and Ofer Zeitouni.
\newblock {\em An introduction to random matrices}, volume 118 of {\em
  Cambridge Studies in Advanced Mathematics}.
\newblock Cambridge University Press, Cambridge, 2010.

\bibitem[CH21]{cheong2021cohen}
Gilyoung Cheong and Yifeng Huang.
\newblock {Cohen--Lenstra distributions via random matrices over complete
  discrete valuation rings with finite residue fields}.
\newblock {\em Illinois Journal of Mathematics}, 65(2):385--415, 2021.

\bibitem[CH25]{cheong2023cokernel}
Gilyoung Cheong and Yifeng Huang.
\newblock The cokernel of a polynomial push-forward of a random integral matrix
  with concentrated residue.
\newblock {\em Mathematical Proceedings of the Cambridge Philosophical Society}, 178(2):229--257, 2025.

\bibitem[CK22]{cheong2022generalizations}
Gilyoung Cheong and Nathan Kaplan.
\newblock {Generalizations of results of Friedman and Washington on cokernels
  of random p-adic matrices}.
\newblock {\em Journal of Algebra}, 604:636--663, 2022.

\bibitem[CL84]{cohen-lenstra}
Henri Cohen and Hendrik~W Lenstra.
\newblock Heuristics on class groups of number fields.
\newblock In {\em Number Theory Noordwijkerhout 1983}, pages 33--62. Springer,
  1984.

\bibitem[CLS23]{cheong2023polynomial}
Gilyoung Cheong, Yunqi Liang, and Michael Strand.
\newblock Polynomial equations for matrices over integers modulo a prime power
  and the cokernel of a random matrix.
\newblock {\em Linear Algebra and its Applications}, 677:1--30, 2023.

\bibitem[FW87]{friedman-washington}
Eduardo Friedman and Lawrence~C Washington.
\newblock On the distribution of divisor class groups of curves over a finite
  field.
\newblock {\em Th{\'e}orie des Nombres/Number Theory Laval}, 1987.

\bibitem[GM20]{GM}
Vadim Gorin and Adam~W. Marcus.
\newblock Crystallization of random matrix orbits.
\newblock {\em Int. Math. Res. Not. IMRN}, (3):883--913, 2020.

\bibitem[Hod23]{hodges2023distribution}
Eliot Hodges.
\newblock The distribution of sandpile groups of random graphs with their
  pairings.
\newblock {\em arXiv preprint arXiv:2311.07078}, 2023.

\bibitem[Hua24]{huang2024cokernel}
Yifeng Huang.
\newblock Cokernels of random matrix products and flag {C}ohen--{L}enstra
  heuristic.
\newblock {\em Forum. Math.}, 2024.
\newblock to appear.

\bibitem[KN22]{kahle2022topology}
Matthew Kahle and Andrew Newman.
\newblock Topology and geometry of random 2-dimensional hypertrees.
\newblock {\em Discrete \& Computational Geometry}, 67(4):1229--1244, 2022.

\bibitem[Lee23]{lee2023joint}
Jungin Lee.
\newblock Joint distribution of the cokernels of random p-adic matrices.
\newblock {\em Forum Mathematicum}, 35(4):1005--1020, 2023.

\bibitem[Lee24]{lee2024mixed}
Jungin Lee.
\newblock Mixed moments and the joint distribution of random groups.
\newblock {\em Journal of Algebra}, 641:49--84, 2024.

\bibitem[Mac]{mac}
Ian~Grant Macdonald.
\newblock {\em Symmetric functions and Hall polynomials}.
\newblock Oxford university press, 1998.

\bibitem[M{\'e}s20]{meszaros2020distribution}
Andr{\'a}s M{\'e}sz{\'a}ros.
\newblock The distribution of sandpile groups of random regular graphs.
\newblock {\em Transactions of the American Mathematical Society},
  373(9):6529--6594, 2020.

\bibitem[M{\'e}s24]{meszaros2024phase}
Andr{\'a}s M{\'e}sz{\'a}ros.
\newblock A phase transition for the cokernels of random band matrices over the
  p-adic integers.
\newblock {\em arXiv preprint arXiv:2408.13037}, 2024.

\bibitem[NS06]{NS}
Alexandru Nica and Roland Speicher.
\newblock {\em Lectures on the Combinatorics of Free Probability}.
\newblock London Mathematical Society Lecture Note Series. Cambridge University
  Press, 2006.

\bibitem[NVP24a]{nguyen2024rank}
Hoi~H Nguyen and Roger Van~Peski.
\newblock Rank fluctuations of matrix products and a moment method for growing
  groups.
\newblock {\em arXiv preprint arXiv:2409.03099}, 2024.

\bibitem[NVP24b]{nguyen2022universality}
Hoi~H. Nguyen and Roger Van~Peski.
\newblock Universality for cokernels of random matrix products.
\newblock {\em Adv. Math.}, 438:Paper No. 109451, 70, 2024.

\bibitem[NW22]{nguyen2022random}
Hoi~H Nguyen and Melanie~Matchett Wood.
\newblock Random integral matrices: universality of surjectivity and the
  cokernel.
\newblock {\em Inventiones mathematicae}, 228(1):1--76, 2022.

\bibitem[Spe93]{Speicher}
Roland Speicher.
\newblock Free convolution and the random sum of matrices.
\newblock {\em Publ. Res. Inst. Math. Sci.}, 29(5):731--744, 1993.

\bibitem[SW22]{sawin2022moment}
Will Sawin and Melanie~Matchett Wood.
\newblock The moment problem for random objects in a category.
\newblock {\em arXiv preprint arXiv:2210.06279}, 2022.

\bibitem[Voi91]{Voi1}
Dan Voiculescu.
\newblock Limit laws for random matrices and free products.
\newblock {\em Invent. Math.}, 104(1):201--220, 1991.

\bibitem[VP21]{van2020limits}
Roger Van~Peski.
\newblock Limits and fluctuations of $p$-adic random matrix products.
\newblock {\em Selecta Mathematica}, 27(5):1--71, 2021.

\bibitem[VP22]{vanpeski2021halllittlewood}
Roger Van~Peski.
\newblock Hall–{L}ittlewood polynomials, boundaries, and $p$-adic random
  matrices.
\newblock {\em International Mathematics Research Notices},
  2023(13):11217--11275, 2022.

\bibitem[VP23a]{van2023local}
Roger Van~Peski.
\newblock Local limits in $ p $-adic random matrix theory.
\newblock {\em arXiv preprint arXiv:2310.12275}, 2023.

\bibitem[VP23b]{van2023reflecting}
Roger Van~Peski.
\newblock {Reflecting Poisson walks and dynamical universality in $ p $-adic
  random matrix theory}.
\newblock {\em arXiv preprint arXiv:2312.11702}, 2023.

\bibitem[Woo17]{W0}
Melanie~Matchett Wood.
\newblock The distribution of sandpile groups of random graphs.
\newblock {\em Journal of the American Mathematical Society}, 30(4):915--958,
  2017.

\bibitem[Woo19]{W1}
Melanie~Matchett Wood.
\newblock Random integral matrices and the {Cohen-L}enstra heuristics.
\newblock {\em American Journal of Mathematics}, 141(2):383--398, 2019.

\end{thebibliography}
\end{document}